\documentclass[a4paper,leqno,11pt]{amsart}

\usepackage{amsfonts,amssymb,verbatim,amsmath,amsthm,latexsym,textcomp,amscd}
\usepackage{latexsym,amsfonts,amssymb,epsfig,verbatim}
\usepackage{amsmath,amsthm,amssymb,latexsym,graphics,textcomp}
\usepackage{mathtools}
\usepackage{paralist}
\usepackage{graphicx}
\usepackage{color}
\usepackage{url}
\usepackage{enumerate}
\usepackage[mathscr]{euscript}
\usepackage{tikz-cd}
\usetikzlibrary{shapes.geometric}
\usepackage{dsfont}
\usetikzlibrary{matrix}
\usepackage{hyperref}

\input xy

\xyoption{all}

\theoremstyle{definition}
\newtheorem{theorem}{Theorem}[section]
\newtheorem{prop}[theorem]{Proposition}
\newtheorem{lemma}[theorem]{Lemma}
\newtheorem{cor}[theorem]{Corollary}
\newtheorem{defn}[theorem]{Definition}

\newtheorem{rmk}[theorem]{Remark}

\newtheorem{exam}[theorem]{Example}

\newtheorem{thm}[theorem]{Theorem}
\newtheorem{notation}[theorem]{Notation}
\newtheorem{subsec}[theorem]{}

\theoremstyle{plain}
\newtheorem*{thma}{Theorem A}
\newtheorem*{thmb}{Theorem B}
\newtheorem*{thmc}{Theorem C}

\theoremstyle{remark}
{\swapnumbers
   \newtheorem{ack}[theorem]{Acknowledgements} }
   
\newenvironment{myeq}[1][]
{\stepcounter{theorem}\begin{equation}\tag{\thetheorem}{#1}}
{\end{equation}}

\newenvironment{mysubsection}[2][]
{\begin{subsec}\begin{upshape}\begin{bfseries}{#2.}
			\end{bfseries}{#1}}
		{\end{upshape}\end{subsec}}

\newcommand{\C}{{\mathbb C}}

\newcommand{\Hyp}{{\mathbb H}}

\newcommand{\Z}{{\mathbb{Z}}}

\newcommand{\cO}{\mathbb O}

\newcommand{\Fib}{\mbox{Fib}}

\newcommand{\Top}{{\mathcal{T}op}}

\newcommand{\F}{\mathbb F}

\newcommand\DD{{\mathcal D}}

\newcommand\FF{{\mathcal F}}
\newcommand\GG{{\mathcal G}}

\newcommand\LL{{\mathcal L}}
\newcommand\MM{{\mathcal M}}

\newcommand\PP{{\mathcal P}}

\newcommand\RR{{\mathcal R}}

\newcommand\PMF{{\PP\kern-2pt\MM\FF}}

\newcommand\PML{{\PP\kern-2pt\MM\LL}}

\newcommand{\fsubd}{\mathrel{{\scriptstyle\searrow}\kern-1ex^d\kern0.5ex}}
\newcommand{\bsubd}{\mathrel{{\scriptstyle\swarrow}\kern-1.6ex^d\kern0.8ex}}
\newcommand{\fsubeq}{\mathrel{\raise-.7ex\hbox{$\overset{\searrow}{=}$}}}
\newcommand{\bsubeq}{\mathrel{\raise-.7ex\hbox{$\overset{\swarrow}{=}$}}}

\newcommand{\tsh}[1]{\left\{\kern-.9ex\left\{#1\right\}\kern-.9ex\right\}}

\renewcommand{\Im}{\operatorname{Im}}

\newcommand{\Ker}{\mbox{Ker}}

\newcommand{\Rank}{\mbox{Rank}}
\newcommand{\inter}{\mbox{int}}

\makeatletter
\@namedef{subjclassname@2020}{%
  \textup{2020} Mathematics Subject Classification}
\makeatother

\title{Sphere fibrations over highly connected manifolds}

\author{Samik Basu}
\address{Stat-Math Unit, Indian Statistical Institute, Kolkata 700108, India.}
\email{samik.basu2@gmail.com, samikbasu@isical.ac.in}
\author{Aloke Kr Ghosh}
\address{Stat-Math Unit, Indian Statistical Institute, Kolkata 700108, India.}
\email{alokekrghosh005@gmail.com}


\begin{document}

\begin{abstract}
 We construct sphere fibrations over $(n-1)$-connected $2n$-manifolds such that the total space is a connected sum of sphere products. More precisely, for $n$ even, we construct fibrations $S^{n-1} \to \#^{k-1}(S^n \times S^{2n-1}) \to M_k$, where $M_k$ is a $(n-1)$-connected $2n$-dimensional Poincar\'{e} duality complex which satisfies $H_n(M_k)\cong \Z^k$, in a localized category of spaces. The construction of the fibration is proved for $k\geq 2$, where the prime $2$, and the primes which occur as torsion in $\pi_{2n-1}(S^n)$ are inverted. In specific cases, by either assuming $n$ is small, or assuming $k$ is large we can reduce the number of primes that need to be inverted. Integral results are obtained for $n=2$ or $4$, and if $k$ is bigger than the number of cyclic summands in the stable stem $\pi_{n-1}^s$, we obtain results after inverting $2$. Finally, we prove some applications for fibrations over $N\# M_k$, and for looped configuration spaces.  
\end{abstract}

\maketitle

\section{Introduction}
From the point of view of classification problems in differential topology, apart from surfaces and $3$-manifolds, the first  systematic results were proved for spheres in the celebrated works of Milnor \cite{Mil56} and Kervaire and Milnor \cite{KeMi63}. In  dimension $4$, classification problems have received a lot of attention from both topologists and geometers. The determination of simply connected $4$-manifolds up to homotopy goes back to the early works of Whitehead \cite{Whi49} and Milnor \cite{Mil58}.  In this case, the simply connected  hypothesis on the $4$-manifold $M$ determines the homology groups up to an integer $k$, given by $H_2(M)\cong \Z^k$, and the homotopy type up to the classification of inner product spaces of rank $k$, given by the intersection form. Conversely, given a non-singular inner product space, the associated cell complex does satisfy Poincar\'{e} duality. However, not all of them are homotopy equivalent to smooth manifolds due to the restrictions proved by Rohlin \cite{Roh52} and Donaldson \cite{Don83}. On the other hand, the topological classification problem for simply connected $4$-manifolds solved by Freedman \cite{Fre82}, does not carry the same restrictions.

A natural generalization of the simply connected $4$-manifolds are the $(n-1)$-connected $2n$-manifolds. For these manifolds, the homology is again determined up to an integer $k$ by $H_n(M)\cong \Z^k$. Their classification have been studied by Wall \cite{Wal62} via the approach of expressing these as a union of handlebodies. The intersection form is no longer sufficient to determine the homotopy type of $M$. Such an $M$ with $H_n(M)\cong \Z^k$ possesses a minimal CW complex structure 
\[ M \simeq (S^n)^{\vee k} \cup_{L(M)} \DD^{2n}, \mbox{ with } L(M)\in \pi_{2n-1}\big( (S^n)^{\vee k}\big).\]
The homotopy group $\pi_{2n-1}\big((S^{n})^{\vee k}\big)$ is computed via the Hilton-Milnor theorem \cite{Hil55} as
\[\pi_{2n-1}\big((S^{n})^{\vee k}\big) \cong  (\pi_{2n-1}(S^n))^{\oplus k} \oplus (\pi_{2n-1} S^{2n-1})^{\oplus \binom{k}{2}}.\] 
The groups $\pi_{2n-1}(S^{2n-1})\cong \Z$ occuring in the above description are mapped to $(S^n)^{\vee k}$ via Whitehead products of the different summands. Moreover, if $n$ is even, the group $\pi_{2n-1}S^n$ contains a $\Z$-summand whose generator may be chosen either as the Whitehead product or the Hopf invariant one classes (which occur only when $n=2$, $4$, or $8$ \cite{Ada60}). The projection of $L(M)$ onto these torsion-free summands are determined directly by the intersection form. 

It is also an interesting question whether given $L(M)\in \pi_{2n-1}\big((S^{n})^{\vee k}\big)$, there is a $(n-1)$-connected $2n$-manifold homotopy equivalent to the cell complex $M$. In this paper, we work around these issues by considering all such cell complexes $M$. These satisfy Poincar\'{e} duality in the sense that there is a degree $2n$ homology class $[M]$ which gives the Poincar\'{e} duality isomorphism via the cap product, and are called Poincar\'{e} duality complexes \cite{Wal67}. We write $\PP \DD_k^m$ for the collection of Poincar\'{e} duality complexes that are $k$-connected and $m$-dimensional. In this notation, the above examples lie in $\PP \DD_{n-1}^{2n}$. 

The expression for $L(M)\in \pi_{2n-1}\big((S^{n})^{\vee k}\big)$ shows that a general homotopical classification will rely on the knowledge of $\pi_{2n-1}S^n$, and thus, is not possible with our current knowledge of the homotopy groups of spheres. As a weaker classification, we consider the homotopy type of the loop space $\Omega M$. If $k=\Rank(H_n(M))\geq 2$, one realizes that the homotopy type of $\Omega M$ depends only on $k$ \cite{BeTh14,BaBa18}. One proves that the loop space is expressable as a weak product of the loop space of spheres which map to $\pi_\ast M$ via Whitehead products. If $k=1$, this is not true as is observed in \cite[\S 4.3]{BaBa18}. 

The splitting results for the loop space of manifolds fall under the general framework of loop space decompositions. Such a decomposition for highly connected manifolds was first proved in \cite{BeWu15} for the $(n-1)$-connected $(2n+1)$-manifolds. There have been a growing interest in results of this type \cite{BeTh14, BaBa18, BaBa19, Bas19, Th22_unpub, HuTh22_unpub}. A general idea for producing loop space decompositions is given in \cite{Th22_unpub}. Given a cofibration sequence $\Sigma A \to E \stackrel{h}{\to} J$ for which $\Omega h$ has a right homotopy inverse, there are equivalences 
\[ \Omega E \simeq \Omega J \times \Omega \Fib(h), ~ \Fib(h)\simeq \Sigma A \rtimes \Omega J,\] 
where $\Fib(h)$ is the homotopy fibre of $h$. While this technique may be applied in many examples, it is not very useful when the rank of the homology of $E$ is small. 

A different view of the loop space decompositions is given by fibre bundles. For example, in the case of $\C P^2$, the usual quotient is part of the principal bundle $S^1 \to S^5 \to \C P^2$ which yields the loop space decomposition $\Omega \C P^2 \simeq S^1 \times \Omega S^5$. Simply connected $4$-manifolds also support principal $S^1$-bundles of the form $S^1 \to \#^{k-1}(S^2 \times S^3) \to M$ where $\Rank(H_2(M))=k\geq 2$ \cite{DuLi05, BaBa15}. The construction of such bundles have many geometric consequences. In the context of loop space decompositions, this implies $\Omega M \simeq S^1 \times \Omega (\#^{k-1}(S^2\times S^3))$. The construction involves a choice of a primitive class in $H^2(M)\cong [M,\C P^\infty]$ using the fact that $\C P^\infty$ is the classifying space for $S^1$-bundles, and the classification of spin $5$-manifolds by Smale \cite{Sma62}. In this paper, we search for generalizations of this construction for highly connected manifolds. 

Let $M_k \in \PP \DD_{n-1}^{2n}$ be a Poincar\'{e} duality complex of dimension $2n$ which is $(n-1)$-connected and $\Rank(H_n(M_k))=k$. Let $E_k=\#^{k-1}(S^n\times S^{2n-1})$. We first observe that the existence of a fibration 
\[ S^{n-1} \to E_k \to M_k\]
puts some restrictions on $n$. As $E_k$ is $(n-1)$-connected, we must have that the map $S^{n-1} \to E_k$ is null-homotopic. Now continuing the homotopy fibration sequence further, we find that $\Omega E_k \to \Omega M_k \to S^{n-1}$ is a principal fibration with a section and so, there is a splitting $\Omega M_k \simeq \Omega E_k \times S^{n-1}$. Therefore, the map $S^{n-1} \stackrel{u}{\to} \Omega M_k$ which is adjoint to a map $S^n \stackrel{\tilde{u}}{\to} M_k$, constructs a summand of $\pi_\ast M_k$ via 
\[ \pi_r(S^{n-1}) \stackrel{E}{\to} \pi_{r+1} S^n \stackrel{\tilde{u}}{\to} \pi_{r+1}(M_k).\] 
This means that $E: \pi_r S^{n-1} \to \pi_{r+1}S^n$ is injective, which does not usually happen. We assume that $n$ is even, and either $n=2$, $4$, or $8$, or that we are working in the category $\Top_{1/2}$, which is the localized category of spaces after inverting $2$. This hypothesis implies that $\Omega S^n \simeq S^{n-1} \times \Omega S^{2n-1}$, so that $E$ is the inclusion of a summand. 

We first notice that the classification results of Smale \cite{Sma62} and Barden \cite{Bar65} for $5$-manifolds are not available in general. Additionally, the spheres are not loop spaces (other than $n=2$, or $4$), so it is not possible to obtain principal fibrations in general. We approach this problem from a homotopy theoretic point of view. The homology of the loop space is an associative algebra via the Pontrjagin product, and for both $M_k$ and $E_k$, $H_\ast \Omega M_k$ and $H_\ast \Omega E_k$ may be computed as tensor algebras modulo a single relation \cite{BeBo15}. It is then possible to produce a map between the associative algebras $H_\ast(\Omega E_k) \to H_\ast(\Omega M_k)$. Now the results of \cite{BaBa18} imply that both $\Omega E_k$ and $\Omega M_k$ are a weak product of loop spaces on spheres, which enable us to construct a map $\Omega E_k \to \Omega M_k$ that realizes the above map on homology. 

The next step is to try to construct a delooping of the map $\Omega E_k \to \Omega M_k$. This is done via obstruction theory using the cell structure of $E_k$ as 
\[ E_k \simeq \big( (S^n)^{\vee k-1} \vee (S^{2n-1})^{\vee k-1}\big) \cup_\phi \DD^{3n-1},\]
where $\phi$ is a sum of Whitehead products. The map $\Omega E_k \to \Omega M_k$ already specifies choices for the map 
\[ (S^n)^{\vee k-1} \vee (S^{2n-1})^{\vee k-1} \to M_k.\]
Once the map $f: E_k \to M_k$ is appropriately constructed, a spectral sequence computation is used to show that $\Fib(f) \simeq S^{n-1}$. The first result that we prove is the following theorem (see Theorem \ref{sphfibloc}) 
\begin{thma}
For $k\geq 2$, let $M_k\in \PP \DD_{n-1}^{2n}$  with $n$ even and $H_n(M;\Z) \cong \Z^k$. After inverting the primes $2$ and those which occur as torsion in $\pi_{2n-1}(S^n)$, there is a  fibration $S^{n-1}\to E_k\to  M_k$ such that $E_k$ is homotopy equivalent to $\#^{k-1} (S^{n}\times S^{2n-1})$. 
\end{thma}

Following this, we try for improved results which reduce the set of primes that are required to be inverted. The best case is when $n=2$, $4$, or $8$, which contains a Hopf invariant one class. In the case $n=2$, one already has a $S^1$-bundle over $M_k \in \PP \DD_1^4$ as stated above. We also point this out from our homotopy theoretic techniques without using the classification results of Smale. For $n=4$, we have $\pi_7(S^4) \cong \Z\oplus \Z/(12)$, so the primes that are required to be inverted in Theorem A are $2$, and $3$. We observe through direct computation (Example \ref{even2}) that there is no principal $S^3$-bundle over certain $M_2\in \PP\DD_3^8$ with total space $S^4\times S^7$. Here, the group structure on $S^3$ is by quaternionic multiplication identifying $S^3$ as the unit quaternions. 
However, we are able to prove integral versions of the sphere fibrations as stated in the following theorem.  (see Theorems \ref{sc4mfld}, \ref{sphfib4})
\begin{thmb}
a) Let $M_k$ be a simply connected $4$-manifold with $H_{2}M_k\cong \Z^k$. Then, there is a principal  $S^{1}$-fibration $S^{1}\to E_k \to M_k$ where $E_k\simeq \#^{k-1}(S^{2}\times S^{3}).$ \\
b) Let $M_k\in \PP \DD_3^8$, that is, $H_4(M_k)= \Z^k$ for $k\geq 2$. Such an $M_k$ supports a $S^3$-fibration $S^{3}\to E_k\to M_k$ with $E_k \simeq \#^{k-1}(S^4 \times S^7)$. 
\end{thmb}
For $n=8$, we have $\pi_{15}(S^8)\cong \Z \oplus \Z/(120)$ and so the primes that are required to be inverted in Theorem A are $2$, $3$, and $5$. Through direct computations, we observe that even for $\cO P^2 \# \overline{\cO P^2}$, it is not possible to construct the fibration $S^7 \to S^8 \times S^{15} \to \cO P^2 \# \overline{\cO P^2}$. However, it appears that one may put down a list of criteria on $M\in \PP \DD_7^{16}$ for which such fibrations do exist. We leave this question open for future research. 

For general $n$, one may increase the value of $k$ to obtain better results for spherical fibrations. The precise bound is given by the number of cyclic summands $r$ in the stable stem $\pi_{n-1}^s$. The fibrations are then obtained in the category $\Top_{1/2}$ if $k>r$. If we further pass to $\Top_{1/2,1/3}$ (that is, also invert the prime $3$), we have a homotopy associative multiplication on $S^{n-1}$ which allows us to construct  spaces $E_2(S^{n-1})\simeq S^{n-1}\ast S^{n-1}$ and  $P_2(S^{n-1})$, and a fibration 
\[ S^{n-1} \to E_2(S^{n-1}) \to P_2(S^{n-1}).\] 
In the category $\Top_{1/2,1/3}$, we prove that the sphere fibrations are obtained as a pullback of the above fibration via a map $M\to P_2(S^{n-1})$. These results are summarized in the following theorem. (see Theorems \ref{sphfibgen}, \ref{sphfibpbk})
\begin{thmc}
a) Let $M_k\in \PP \DD_{n-1}^{2n}$  with $H_nM\cong \Z^k$, and $n>8$. Let $r$ be the number of  cyclic odd torsion summands of $\pi_{n-1}^s$. If $k>r$,  after inverting $2$ there is a fibration $E_k \to M_k$ with fibre $S^{n-1}$ where $E_k\simeq \#^{(k-1)}(S^{n}\times S^{2n-1})$.\\
b) After inverting $2$ and $3$, the fibration $E_k\to M_k$ is a homotopy pullback of $M_k \to P_2(S^{n-1})\leftarrow E_2(S^{n-1})$ for a suitable map $M_k\to P_2(S^{n-1})$.  
\end{thmc}

As applications of the spherical fibrations constructed here we consider connected sums $N\# M_k$ with $M_k$ as above, and $N$ a general simply connected $2n$-manifold (more precisely, $N\in \PP \DD_1^{2n}$). Using the techniques in \cite{Th22_unpub}, we may observe that a loop space decomposition may be obtained using the fact that the attaching map of the top cell of $M_k$ is inert. The condition {\it inert} means that the map $(S^n)^{\vee k} \to M_k$ has a right homotopy inverse after taking the loop space (that is $\Omega \big((S^n)^{\vee k}\big) \to \Omega M_k$ has a right homotopy inverse). 

We present a fresh view for these loop space decompositions (Theorems \ref{loopspdecconnsum}, \ref{connloopdechigh}, \ref{connloopdec24}). We input our fibrations into the arguments in \cite{HuTh22_unpub} to realize fibrations over $N\# M$ with total space $G_\tau(N)\# E_k$ (Proposition \ref{pdident}). The manifold $G_\tau(N)$ is a union of $N_0 \times S^{n-1}$ and $S^{2n-1}\times \DD^n$ \eqref{Gtaudefn} using an equivalence $S^{2n-1}\times S^{n-1} \to S^{2n-1}\times S^{n-1}$ associated to a map $\tau$ from $S^{2n-1}$ to the space of homotopy equivalences of $S^{n-1}$. Earlier observations about connected sums of sphere products \cite[Theorem B]{BaBa18} imply that the attaching map of the top cell is inert. Thus, the homotopy type of $\Omega(G_\tau(N)\# E_k)$ depends only on $G_\tau(N)-\ast$. We identify this to be $N_0\rtimes S^{n-1}$ and is thus independent of $\tau$ (Proposition \ref{gtauminpt}). Finally, we point out that the loop space decompositions also yield results for the loop space of the configuration spaces of $N\# M_k$ (Theorem \ref{configloopdec}).

\begin{mysubsection}{Organization}
In \S \ref{algres}, we prove algebraic results over the tensor algebra and the graded Lie algebras, that are later used in the paper. More explicitly, we construct the algebra map between the homology of the loop spaces of  $E_k$ and $M_k$.  In  \S \ref{loccat}, we prove Theorem A. This section also contains the spectral sequence argument used to identify the homotopy fibre of $E_k \to M_k$ used throughout the paper. In \S \ref{lowdim}, we point out the results in low dimensions, mainly concentrating on the cases $n=2$ and $n=4$. We prove the results for general $n$ in  \S \ref{highbetti}, where a proof of Theorem C is discussed. The applications to fibrations over $N\# M$ are discussed in \S \ref{appln}.
\end{mysubsection}

\begin{notation} 
The following notations are used throughout the document

\begin{itemize}
\item Throughout the document, $n$ is fixed to be an even integer.
\item All manifolds are assumed to be closed, compact and smooth unless otherwise stated.
\item $\iota_n$ denotes the homotopy class of the identity map $S^n \to S^n$. 
\item For an element $\chi \in \pi_{n+k}(S^n)$, $\chi_{(m)}$ denotes the suspension $E^{m-n}(\chi)\in \pi_{m+k}(S^m)$ for $m\geq n$. 
\item $\pi_m^s$ is the $m^{th}$ stable stem. 
\item For a graded vector space $V$, $T(V)$ denotes the tensor algebra of $V$. 
\item For a graded vector space $V$, $F(V)$ denotes the free graded Lie algebra  on $V$. For an element $\LL\in F(V)$, $L(V,\LL)$ denotes $F(V)/(\LL)$ where $(\LL)$ is the Lie algebra ideal of $F(V)$ generated by $\LL$. 
\item For a space $X$, $\rho$ denotes the composite $\pi_m(X) \stackrel{\cong}{\to} \pi_{m-1}(\Omega X) \to H_{m-1}(\Omega X)$. 
\item $\PP \DD_k^m$ is the collection of $k$-connected Poincar\'{e} duality complexes of dimension $m$. 
\item $T_n$ denotes $\{2\} \cup \{p \mid \exists \mbox{ non-trivial } p\mbox{-torsion in } \pi_{2n-1}(S^n) \}$. The notation $R_n$ stands for the ring $\Z[\frac{1}{p}\mid p\in T_n]$. 
\item The notation $M_k \in \PP \DD_{n-1}^{2n}$ stands for a $(n-1)$-connected Poincar\'{e} duality complex $M_k$ with $H_n(M_k)\cong \Z^k$. 
\item For $M_k \in \PP \DD_{n-1}^{2n}$, the attaching map of the top cell is denoted by $L(M_k):S^{2n-1} \to (S^n)^{\vee k}$. This is expressed in the general form 
\begin{myeq}\label{LMform}
L(M_k)= \sum_{1\leq i < j \leq k} g_{i,j} [\alpha_i, \alpha_j] + \sum_{i=1}^k g_{i,i} \phi_i + \sum_{i=1}^k l_i \psi_i.  
\end{myeq}
where $\alpha_i$ is the inclusion of the $i^{th}$ wedge summand, $\phi_i=\alpha_i \circ \phi$, and $\psi_i = \alpha_i \circ \psi(i)$  with $\phi, \psi(i) \in \pi_{2n-1}S^n$, the Hopf invariant $H(\psi(i))=0$, and $\phi= \eta, \nu,$ or $\sigma$ if $n=2,4,$ or $8$, and $\phi=[\iota_n,\iota_n]$ otherwise. The matrix $\big( (g_{i,j})\big)$ is invertible over $\Z$.  
\item The notation $a_i\in H_{n-1}(\Omega M)$ is defined as $\rho(\alpha_i)$.
\item $E_k$ denotes a space homotopy equivalent to $\#^{k-1}(S^n\times S^{2n-1})$. This depends also on $n$, but we drop it from the notation as $n$ is a fixed even integer. 
\item The $(3n-2)$-skeleton of $E_k$ is a wedge of spheres $(S^n)^{\vee k-1} \vee (S^{2n-1})^{\vee k-1}$. The notation $\mu_i$ stands for the inclusion of the $i^{th}$-copy of $S^n$, and $\delta_i$ stands for the inclusion of the $i^{th}$-copy of $S^{2n-1}$. The loop space homology generators are denoted $\tilde{\mu_i}$, and $\tilde{\delta_i}$, defined as $\tilde{\mu_i}=\rho(\mu_i)$, and $\tilde{\delta_i}=\rho(\delta_i)$. 
\item The notation $X \rtimes Y$ stands for $X\times Y/( \ast\times Y)$, and $X\ltimes Y$ stands for $X\times Y/(X \times \ast)$.
\item The space $F(n)$ is the $H$-space of homotopy equivalences of $S^{n-1}$. 
\end{itemize}
\end{notation}

\begin{ack} 
The first author would like to thank David Blanc for assistance with explicit formula for Whitehead products in low degrees used in the paper. The second author would like to thank Debanil Dasgupta and Pinka Dey for useful discussions.
\end{ack}

\section{Constructing maps between loop space homology algebras}\label{algres}
In this section, we construct a map from the homology of the loop space of a connected sum of copies of $S^n\times S^{2n-1}$ to that of the loop space of a highly connected Poincar\'{e} duality complex. We use the fact that the latter is a quadratic algebra with a single relation which in turn comes from a non-singular intersection form. 

\begin{mysubsection}{Some algebraic results}
Let $V$ be a free module over a principal ideal domain $R$ (in our applications $\Z[\frac{1}{p_1},\cdots, \frac{1}{p_r}]$ for a finite set of primes $\{p_1,\cdots, p_r\}$) of finite rank $k$, and suppose $\alpha : V \to R$ is a non-zero linear function. Let $\LL$ be a symmetric $2$-tensor (that is an element of $Sym^2(V)=(V\otimes V)^{\Sigma_2}$) which is invertible (that is, with respect to any basis the corresponding $k\times k$ matrix with coefficients from $R$ is invertible). We think of $V$ as a graded vector space, also note that the associative algebra $T(V)$ has a graded Lie bracket given by 
\begin{myeq}\label{bracktv}
 [v,w] = v\otimes w - (-1)^{|v||w|}w\otimes v \mbox{ for all } v, w \in T(V) .
 \end{myeq}

\begin{prop}\label{alglem}
With $V$ concentrated in a single grading $m$ for $m$ odd, $\alpha$ and $\LL$ as above, for any basis $v_1, \cdots, v_{k-1}$ of $\Ker(\alpha)$, there are $w_1, \cdots, w_{k-1} \in V\otimes V$ such that 
\[ \begin{aligned} \mbox{ 1. } & \sum_{j=1}^{k-1} [v_j,w_j]= 0 \pmod{\LL}. \\ 
 \mbox{ 2. } & \{w_1,\cdots, w_{k-1} \}  \mbox{ projects to a basis of } V\otimes V/\big(R\{\LL\} + V\otimes \Ker(\alpha)\big). 
\end{aligned} \] 
\end{prop}

\begin{proof}
Given a basis $v_1, \cdots, v_{k-1}$ of $\Ker(\alpha)$, pick $v_k$ such that $v_1, \cdots, v_k$ is a basis of $V$. This is possible as the image of $\alpha$ is a principal ideal $(b)$ of $R$ , and we may pick $v_k$ such that $\alpha(v_k)=b$. As the collection $\{v_i \otimes v_j\}$ is a basis of $V\otimes V$, we have an expression 
\begin{myeq}\label{expl}
 \LL = \sum_{i=1}^k \sum_{j=1}^k g_{i,j} v_i \otimes v_j,
 \end{myeq}
for a symmetric invertible matrix $\big((g_{i,j})\big)$ over $R$. Define $w_i$ by 
\begin{myeq} \label{loopbi}
w_i = \sum_{j=1}^{k-1}g_{i,j}[v_j,v_k] + g_{i,k}v_{k}\otimes v_{k}.
\end{myeq} 
Note that a basis of the free $R$-module $V\otimes V/\big( V\otimes \Ker(\alpha)\big)$ is given by the images of the elements $v_j\otimes v_k$ for $1\leq j \leq k$. It is clear that the coefficients of the $w_i$ in terms of this basis are those of the first $(k-1)$ columns of the matrix  $\big((g_{i,j})\big)$, with the last column corresponding to $\LL$. This proves 2. For the statement 1, we compute using the graded Jacobi identity 
\[ [x,[y,z]]= [[x,y],z] +(-1)^{|x||y|}[y,[x,z]],\]
and the identity  \cite[\S 8.1]{Nei10}
\[[y,x\otimes x]=[[y,x],x],\]
for odd degree classes $x$. We have
\[\begin{aligned} 
\sum_{i= 1}^{k-1}[v_i, w_i] & = \sum_{i= 1}^{k-1}\Big(\sum_{j= 1}^{k-1}g_{i,j}[v_i,[v_j,v_k]] + g_{i,k} [v_i, v_k\otimes v_k] \Big) \\ 
&= \sum_{1\leq i < j \leq k-1} g_{i,j}[[v_i,v_j],v_k] + \sum_{i=1}^{k-1} g_{i,i} [v_i\otimes v_i, v_k] + \sum_{i=1}^{k-1} g_{i,k}[[v_i,v_k],v_k] \\
&= [\LL,v_k] - g_{k,k}[v_k\otimes v_k, v_k]\\ 
&= [\LL,v_k].
\end{aligned}\]
The last step is true as $[v_k\otimes v_k, v_k]=0$. 
\end{proof}

We carry forward the analogy in Proposition \ref{alglem} further using graded Lie algebras. First recall the definition of a graded Lie algebra \cite{Nei10}.  A graded Lie algebra over a ring $\RR$ in which $2$ is not invertible carries an extra squaring operation on odd degree classes to encode the relation $x^2 = \frac{1}{2} [x,x]$ whenever $|x|$ is odd. 
\begin{defn}
 A graded Lie algebra $L=\oplus L_i$ is a graded $\RR$-module together with a Lie bracket
$$[~ ,~ ] : L_i \otimes_{\RR} L_j \to  L_ {i+j}$$
and a quadratic operation called {\it squaring} defined on odd degree classes
$$(~)^2 : L_{2k+1} \to L_{4k+2}.$$
These operations are required to satisfy the identities
\begin{align*}
[x, y] & = -(-1)^{\deg(x)\deg(y)}[y, x]\,\,\,(x\in L_i, y\in L_j),\\
[x, [y, z]] & = [[x, y], z] + (-1)^{\deg(x)\deg(y)} [y, [x, z]] \,\,\,(x\in L_i, y\in L_j, z\in L_k),\\
(ax)^ 2 & = a^2 x^2 \,\,\,(a\in \RR, x\in L_{2k+1}),\\
(x + y)^2 & = x^2 + y^2 + [x, y], \\
[x, x] & = 0 \,\,\,(x\in L_{2i}),\\
2x^2 & = [x, x],\,\,[x, x^2] = 0 \,\,\,\hfill(x\in L_{2k+1}),\\
[y, x^2] & = [[y, x], x] \,\,\,(x\in L_{2k+1}, y\in L_i).
\end{align*}
\end{defn} 

\begin{exam}
Note that $T(V)$ is a graded Lie algebra with the Lie bracket described in \eqref{bracktv}. For $|u|$ odd, $(u)^2$ is defined to be $u \otimes u$. The identities above are easily verified. A symmetric $2$-tensor $\LL$ which is expressed in the form \eqref{expl} may be written as 
\[\sum_{i=1}^k \sum_{j=1}^k g_{i,j} v_i \otimes v_j = \sum_{1\leq i < j \leq k} g_{i,j}[v_i,v_j] + \sum_{i=1}^k g_{i,i} v_i \otimes v_i.\]
Therefore, $\LL$ belongs to the sub-Lie algebra of $T(V)$ generated by $V$ if  $V$ is concentrated in odd degree as in the hypothesis of Proposition \ref{alglem}. 
\end{exam}

\vspace{0.5cm}

It is possible to derive a Poincar\'e-Birkhoff-Witt theorem for graded Lie algebras under the extra assumption that the underlying module is free over $\RR$. 
\begin{theorem}\label{PBW} \cite[Theorem 8.2.2]{Nei10} If $L$ is a graded Lie algebra over $\RR$  which is a free $\RR$-module in each degree, then $U (L)$ is isomorphic to the symmetric algebra on $L$. In terms of the multiplicative structure, the symmetric algebra on $L$ is isomorphic to the associated graded of $U(L)$ with respect to the length filtration induced on $U(L)$. 
\end{theorem}

We also note the following result which states that the graded Lie algebra injects into the universal enveloping Lie algebra. 
\begin{prop}\label{grinj} \cite[Theorem 2.21]{BaBa18}
Suppose that $\RR$ is a Principal Ideal Domain. Let $L$ be a graded Lie algebra over $\RR$ such that $L_n$ is finitely generated for every $n$. Let $U(L)$ be its universal enveloping algebra. Then the natural map $\iota:L \to U(L)$ is injective. 
\end{prop}

The graded Lie algebra of interest is $L(V,\LL)$ which is defined to be the graded Lie algebra $F(V)/(\LL)$. The notation $F(V)$ stands for the free Lie algebra generated by $V$ and $\LL$ being a symmetric $2$-tensor, lies in $F(V)$ (as $V$ is concentrated in odd degree). We may express the graded Lie algebra as 
\[ L(V,\LL) \cong V \oplus \Big[[V,V] + (V)^2 \Big]/(\LL) \oplus \cdots \]

\begin{exam}
Note that $U(F(V))\cong T(V)$, and \cite[Proposition 2.9]{BaBa18} implies that $U(L(V,\LL))\cong T(V)/(\LL)$.
\end{exam}
\vspace{0.5cm}

Let $\dim(V)=k$. For a $(k-1)$-dimensional summand $W$ of $V$, write   
\[ L^W(V,\LL) \cong W \oplus \Big[[V,V] + (V)^2 \Big]/(\LL) \oplus \cdots \]
which becomes a Lie subalgebra of $L(V,\LL)$. We note that Proposition \ref{alglem} actually identifies the Lie algebra $L^W(V,\LL)$. 
\begin{prop} \label{alglem2}
Given any basis $v_1,\cdots, v_{k-1}$ of $W$, there are $w_1,\cdots, w_{k-1}$ satisfying the conditions of Proposition \ref{alglem} such that the map $F(v_1,\cdots, v_{k-1}, w_1,\cdots, w_{k-1}) \to L(V,\LL)$ induces an isomorphism of graded Lie algebras
\[ \frac{F(v_1,\cdots, v_{k-1},w_1,\cdots, w_{k-1})}{(\sum_{i=1}^{k-1} [ v_i, w_i])} \cong L^W(V,\LL).\]
\end{prop}

\begin{proof}
Let $\FF$ stand for the left hand side of the equation in the statement of the Proposition. We wish to show that $\FF \cong L^W(V,\LL)$. We note that the universal enveloping algebra of $\FF$ is $T(v_1,\cdots, v_{k-1}, w_1,\cdots, w_{k-1})/(\sum [v_i,w_i])$, and the universal enveloping algebra for $L(V,\LL)$ is $T(V)/\LL$. It follows from Proposition \ref{grinj} that both $\FF$ and $L^W(V,\LL)$ are free of finite rank in each grading. Our first observation is that the  Poincar\'{e}-Birkoff-Witt theorem (Theorem \ref{PBW}) implies that the ranks in each degree are the same. We then show that the map is surjective which will complete the proof.

Let $V$ be concentrated in grading $n-1$ which is odd, and let $l_d$ denote the degree $d$ part of $L(V,\LL)$. The symmetric algebra on $L(V,\LL)$ then has Poincar\'{e} series 
\[\frac{\prod_{d ~ odd }(1+ t^{d})^{l_{d}}}{\prod_{d ~ even}(1- t^{d})^{l_{d}}}. \]
By Theorem \ref{PBW}, this is the Poincar\'{e} series of the universal enveloping algebra $T(V)/(\LL)$. As $V$ is concentrated in a single degree $n-1$, and $\LL$ is a quadratic element it follows from \cite[\S 4.4]{BaBa18} that 
\[ \frac{\prod_{d ~ odd }(1+ t^{d})^{l_{d}}}{\prod_{d ~ even}(1- t^{d})^{l_{d}}} = \frac{1}{1-kt^{n-1}+t^{2n-2}}.\]
Analogously, for $\FF$, we let $f_d$ denote the degree $d$ part of $\FF$, and we apply the techniques from \cite[\S 5.2]{BaBa18} to deduce 
\[ \frac{\prod_{d ~ odd }(1+ t^{d})^{f_{d}}}{\prod_{d ~ even}(1- t^{d})^{f_{d}}} = \frac{1}{1-(k-1)t^{n-1}- (k-1)t^{2n-2}+t^{3n-3}}.\]
We now have the factorization
\[ 1-(k-1)t^{n-1}- (k-1)t^{2n-2}+t^{3n-3} = (1+t^{n-1})(1-kt^{n-1} + t^{2n-2}),\]
which implies 
\[ f_d = \begin{cases} l_d &\mbox{ if } d\neq n-1 \\
                        l_d -1 &\mbox{ if } d=n-1 . \end{cases} \]
This implies that the degree-wise rank of $\FF$ matches that of $L^W(V,\LL)$. 

We now complete the proof by showing that the map 
\[ F(v_1,\cdots, v_{k-1}, w_1,\cdots, w_{k-1}) \to L^W(V,\LL)\]
is surjective. We choose $\alpha$ such that $\Ker(\alpha)=W$ in the notation of Proposition \ref{alglem}, and choose $v_k$ so that $v_1,\cdots, v_k$ is a basis for $V$, and proceed by induction on the length of a bracket $r$, to show that any element of $L^W(V,\LL)$ of the form $[[\cdots [l_1, l_2],\cdots], l_r]$ belongs to the image (modulo $\LL$), where each $l_i$ is one of the basis elements $v_j$. It is enough to show this as, if $\lambda$ lies in the image then so does $\lambda^2$, and a bracket whose entries contain squares may be rewritten in terms of those without the squares using the identity $[y,x^2] = [[y,x],x]$, and the Jacobi identity.  For $r=1$, we are done by the choice that $v_1,\cdots, v_{k-1}$ is a basis of $W$. For $r=2$, we are done by the fact that $W\otimes V$, $w_1, \cdots, w_{k-1}$ and $\LL$ together form a basis of $V\otimes V$ by 2) of Proposition \ref{alglem}.   

In the general case, the proof follows from the induction hypothesis if $l_r=v_j$ for $j<k$. We only need to consider $l_r=v_k$. Now the term of interest is $[l_{r-1}, v_k]$ where $l_{r-1}$ lies in the image. Here, we have that the $l_{r-1}$ is the sum of iterated brackets on the $v_j$ and $w_j$, so it suffices to show by the Jacobi identity that $[w_j, v_k]$ is thus representable. From the formula of $w_j$, we have 
\[ [w_j, v_k] = \sum_{l<k} g_{j,l} \big[[v_l,v_k],v_k] + g_{j,k}[v_k^2,v_k]= \sum_{l<k} g_{j,l}[v_l, v_k^2].\]
As $v_k^2 $ lies in $V\otimes V$, the $r=2$ argument expresses this in the image, and thus this expression belongs to the image. The proof is thus complete.   
\end{proof}

\end{mysubsection} 

\begin{mysubsection} {The homology of highly connected Poincar\'{e} duality complexes}
Recall that a Poincar\'{e} duality complex of dimension $r$ is a cell complex, together with a homology class in degree $r$, the cap product with which induces Poincar\'{e} duality as in a manifold of dimension $r$. We write $\PP \DD_k^r$ to be the collection of Poincar\'{e} duality complexes of dimension $r$ that are $k$-connected.  Let $M\in \PP \DD_{n-1}^{2n}$ with $n$ even. The Poincar\'{e} duality condition guarantees that $H_n(M)\cong \Z^k$ for some $k$, and the homology is of the form 
\begin{myeq} \label{homM}
 H_i(M)= \begin{cases} \Z & \mbox{ if } i=0, 2n \\ 
                          \Z^k & \mbox{ if } i=n \\ 
                       0 & \mbox{ otherwise.} \end{cases} 
\end{myeq}
We write $M_k$ for an element of $\PP \DD_{n-1}^{2n}$ having the homology described in \eqref{homM}. Now a minimal cell structure \cite{Hat01} on the space $M_k$ implies the pushout square
\begin{myeq}\label{poM}
\xymatrix{
S^{2n-1}\ar[r]^{i} \ar[d]^{L(M_k)} & \DD^{2n} \ar[d] \\
(S^{n})^{\vee_k}  \ar[r] & M_k,}
\end{myeq}
where $L(M_k)\in \pi_{2n-1}\big((S^n)^{\vee k}\big)$. Let $\alpha_i$ denote the inclusion of the $i^{th}$-copy of $S^n$ in $M_k$. Using Hilton's theorem \cite{Hil55}, one has the decomposition  
\[\pi_{2n-1}\big((S^{n})^{\vee k}\big)= \big(\pi_{2n-1}(S^{n})\big)^{\oplus k}\oplus \big(\pi_{2n-1}(S^{2n-1})\big)^{\oplus \binom{k}{2}}.\]
The first factor in the above decomposition is induced by the inclusion of the spherical wedge summands, and the second factor by the Whitehead products of a choice of two summands. We note down the primes appearing in the torsion subgroup of $\pi_{2n-1}(S^n)$ in the following notation. 
\begin{notation}\label{Tn}
Define $T_n= \{2 \} \cup \{ p \mid p \mbox{ prime and } \exists \mbox{ non-trivial } p\mbox{-torsion in } \pi_{2n-1}(S^n) \}$. 
\end{notation}

A key step in the computations of this paper is the loop space homology of $M_k$, that is, $H_\ast \Omega M_k$. This is a (associative) ring, which is a quadratic algebra if $k\geq 2$ \cite{BaBa18}. More precisely, let $a_i\in H_{n-1}(\Omega M_k)$ denote the Hurewicz image of the  adjoints of the $\alpha_i: S^n \to M_k$, and $l(M_k)$ denote the image of $L(M_k)$ under the composite  
$$\rho \colon \pi_{2n-1}\big((S^{n})^{\vee_{k}}\big)\xrightarrow{\cong} \pi_{2n-2}\Omega\big((S^{n})^{\vee_{k}}\big)\to H_{2n-2}\Big(\Omega\big((S^{n})^{\vee k}\big) \Big).$$
The classes $a_i$ serve as algebra generators of $H_\ast \Omega M_k$ and in their terms $l(M_k)$ may be expressed as 
$$l(M_k)= \sum_{i,j}-g_{i,j}a_{i}\otimes a_{j}.$$
To figure out the sign, observe that 
\[ \rho([\alpha_i,\alpha_j])=(-1)^{|\alpha_i|-1}[a_i,a_j] = -[a_i,a_j] \mbox{ as $n$ is even  \cite{Sam53}. }\]
The matrix $\big((g_{i,j})\big)$ is inverse of the matrix of the intersection form of $M_k$, so in the notation of Proposition \ref{alglem}, this is a symmetric, invertible $2$-tensor. The homology of the loop space may be computed as \cite{BaBa18, BeBo15}
\begin{myeq}\label{loophomM}
H_\ast(\Omega M_k) \cong T(a_1, \cdots, a_k)/\big(l(M_k)\big).
\end{myeq}
\end{mysubsection} 

\begin{mysubsection}{Loop space homology of a connected sum of sphere products}
A connected sum of sphere products has the form $T=S^{k_1}\times S^{n-k_1}  \# \cdots \# S^{k_r}\times S^{n-k_r}$. This is a pushout
$$
\xymatrix{
S^{n-1} \ar[rr] \ar[dd]^ {\sum_{i= 1}^r [\iota_{k_{i}}, \iota_{n-k_i}]} & & \DD^n \ar[dd] \\
\\
\vee_{1\leq i \leq r} S^{k_i}\vee S^{n-k_i}  \ar[rr]^-{{\vee_{i= 1}^{k-1}(\mu_{i} \vee \delta_{i})}} && S^{k_1}\times S^{n-k_1}  \# \cdots \# S^{k_r}\times S^{n-k_r}. }
$$
In this expression $\mu_i : S^{k_i} \to T$ and $\delta_i : S^{n-k_i}\to T$ denotes the inclusion of the various factors. In these terms the loop space homology of $T$ is given by \cite{BaBa18}
\begin{myeq}\label{loophomsphprod}
H_\ast \Omega T \cong T(\tilde{\mu_1}, \tilde{\delta_1}, \cdots, \tilde{\mu_r}, \tilde{\delta_r})/\Big(\sum_{i=1}^r [\tilde{\mu_i}, \tilde{\delta_i}]\Big), 
\end{myeq}
where the superscript $\sim$ is used to denote the adjoint of the class in loop space homology. In this paper, the connected sum $T$ used is of the form $\#^{k-1} (S^n \times S^{2n-1})$. We retain the notation $\mu_i, \delta_i$ as above, and we have the pushout 
\begin{myeq}\label{connsumsph}
\xymatrix{
S^{3n-2} \ar[rr] \ar[d]_{\Sigma_{i= 1}^{k-1}[\iota^{n}_{i}, \iota^{2n-1}_{i}]} && \DD^{3n-1} \ar[d] \\
(S^{n}\vee S^{2n-1})^{\vee_{k-1}}  \ar[rr]^{{\vee_{i= 1}^{k-1}(\mu_{i} \vee \delta_{i})}} && \#^{(k-1)}(S^{n}\times S^{2n-1}) .}
\end{myeq}
The homology of the loop space is given by \eqref{loophomsphprod}. Now we look at Proposition \ref{alglem} from the perspective of the connected sum of sphere products above. We may map the generators $\tilde{\mu_i}$ of the loop space homology to the generators $a_i$ of $H_\ast \Omega M$ for $1\leq i \leq k-1$. Proposition \ref{alglem} now tells us where to send the other generators $\tilde{\delta_i}$ to obtain an algebra map from $H_\ast \Omega T \to H_\ast \Omega M$. We summarize this in the following algebraic result. 
\begin{prop}\label{algmap}
Given a basis $a_1, \cdots, a_k$ of $H_{n-1} \Omega M_k$,  there is a map of associative algebras $H_\ast \Omega \#^{k-1} (S^n\times S^{2n-1}) \to H_\ast \Omega M_k$ which sends $\tilde{\mu_i}$ to the classes $a_i$ for $1\leq i \leq k-1$.
\end{prop}

Our efforts in this paper involve \\
1) Realize the above algebra map by a map of spaces $\#^{k-1} (S^n\times S^{2n-1}) \to M_k$ for $k\geq 2$. \\
2) Identify the homotopy fibre of the corresponding map. 

We show that it is possible to achieve 1) after inverting the primes in $T_n$ or if the value of $k$ is large. Once this is achieved the homotopy fibre is shown to be homotopy equivalent to $S^{n-1}$ by a spectral sequence argument. 
\end{mysubsection}

\section{Sphere fibrations in a localized category}\label{loccat}
The objective of this section is to construct sphere fibrations over highly connected Poincar\'{e} duality complexes, after inverting finitely many primes. In this case, we also have a specific understanding of the primes that need to be inverted as the set $T_n$ of Notation \ref{Tn}. In the entire section, we work in the localized category of spaces in which the primes in $T_n$ are inverted. We write $R_n= \Z[\{\frac{1}{p} \mid p\in T_n \}]$, and the homology computations throughout are taken with $R_n$-coefficients. For a space $X$, recall that $\rho$ is the map 
\[\rho : \pi_n(X) \cong \pi_{n-1}(\Omega X) \to H_{n-1}(\Omega X).\]
Recall the equation $\rho([\gamma_1, \gamma_2]) = (-1)^{|\gamma_1|-1}[\rho(\gamma_1), \rho(\gamma_2)]$ \cite{Sam53}. We prove the following main theorem.
\begin{theorem}\label{sphfibloc3}
For $k\geq 2$, let $M_k\in \PP \DD_{n-1}^{2n}$  with $n$ even, and $H_n(M_k;\Z) \cong \Z^k$. After inverting the primes in $T_n$, and the prime $3$, there is a  fibration $S^{n-1}\to E_k\to  M_k$ such that $E_k$ is homotopy equivalent to $\#^{k-1} (S^{n}\times S^{2n-1})$. 
\end{theorem}
\begin{proof}
Our strategy is to construct a map $f: E_k \to M_k$ such that the homotopy fibre of $f$ is homotopy equivalent to $S^{n-1}$. We do this via the pushout description of $E_k$ in \eqref{connsumsph}. 

For $1\leq i \leq k-1$, we define $f$ on the $i^{th}$-factor $\mu_i$ as the map $\alpha_i : S^n \to M_k$. It follows that the class $\tilde{\mu_i}$ in loop space homology maps to the class $a_i$, as $\rho(\alpha_i)=a_i$. Let $\big((g_{i,j})\big)$ be as in Proposition \ref{alglem}. The $i^{th}$-factor $\delta_i$ is mapped by $f$ in accordance with \eqref{loopbi} to  
\begin{myeq}\label{betai}
\beta_{i} = \sum_{j=1}^{k-1} g_{i,j}[\alpha_{j}, \alpha_{k}] + \frac{1}{2}g_{i,k}[\alpha_{k}, \alpha_{k}] 
\end{myeq}
 which belongs to the image of 
\[\pi_{2n-1} \big( (S^n)^{\vee k} \big) \otimes R_n \to \pi_{2n-1}( M_k) \otimes R_n \]
as $2\in T_n$.  We write $\rho(\beta_i)=-b_i$, anf it follows that in loop space homology $\tilde{\delta_i}$ maps to the class $-b_i$. This defines 
\[ f : (S^n)^{\vee k-1} \bigvee (S^{2n-1})^{\vee k-1} \to M_k,
\]
and to extend $f$ all the way to $E_k$, we require to show that the attaching map of the $(3n-1)$-cell is mapped to $0$ by $f_\ast$. From Proposition \ref{alglem}, we have that 
\[(\Omega f)_\ast \Big(\sum_{i=1}^{k-1} [\tilde{\mu_i}, \tilde{\delta_i}] \Big) =- \sum_{i=1}^{k-1} [a_i,b_i] = 0 .\]
This implies 
\[ f_\ast  \Big(\sum_{i=1}^{k-1} [\mu_i, \delta_i] \Big) \in \Ker(\rho).\]
Recall the definition of $L(M_k)$ in \eqref{poM}. From the definition of $T_n$, we see that the map 
\[ \pi_{2n-1}\big((S^n)^{\vee k}\big) \otimes R_n \stackrel{\rho}{\to} H_{2n-2}\Big(\Omega\big((S^n)^{\vee k} )\big);R_n\Big) \cong T_{R_n}(a_1,\cdots,a_k)\]
is injective. We now write 
\[ \overline{L}(M_k)= \sum_{i<j} g_{i,j} [\alpha_{i}, \alpha_{j}] +\sum_{i}\frac{1}{2}g_{i,i}[\alpha_{i},\alpha_{i}], \]
and observe that $\rho(L(M_k))= \rho(\bar{L}(M_k))=-l(M_k)$. This follows from the equation $\rho([\gamma_1, \gamma_2]) = (-1)^{|\gamma_1|-1}[\rho(\gamma_1), \rho(\gamma_2)]$ \cite{Sam53}.  Moreover, we have the equation
\[
\begin{aligned} 
 f_\ast  \Big(\sum_{i=1}^{k-1} [\mu_i, \delta_i] \Big) &=  \sum_{i=1}^{k-1} [\alpha_i, \beta_i] \\
                                   & = [\bar{L}(M_k),\alpha_k] + \frac{1}{2}g_{k,k}[[\alpha_k,\alpha_k],\alpha_k]
\end{aligned}
\]
from a computation analogous to the proof of Proposition \ref{alglem} via the Jacobi identity for Whitehead products. As the prime $3$ is inverted, we also have that $[[\alpha_k,\alpha_k],\alpha_k]=0$. Now we apply the injectivity of $\rho$ after tensoring with $R_n$ to replace $\bar{L}(M_k)$ with $L(M_k)$ and obtain, 
\[  f_\ast  \Big(\sum_{i=1}^{k-1} [\mu_i, \delta_i] \Big) = [L(M_k), \alpha_k] = 0 \]
in $\pi_{3n-2}(M_k)\otimes R_n$. This constructs the map $f:E_k \to M_k$. 
The homotopy fibre $\Fib(f)$ of $f$ has the same homology as $S^{n-1}$ by the spectral sequence argument of Proposition \ref{spseqarg}, which completes the proof of the theorem as all the spaces are simply connected.
\end{proof}

We now show that inverting $3$ is not necessary in Theorem \ref{sphfibloc3}. We first consider the case $k=2$ in the following example. 

\begin{exam}\label{sphk=2}
Let $M_2\in \PP \DD_{n-1}^{2n}$ with $H_n(M)=\Z^2$. We wish to construct a fibration $S^{n-1} \to S^n\times S^{2n-1} \to M_2$ after inverting the primes in $T_n$. Following the proof of Theorem \ref{sphfibloc3}, we are able to do this if under a choice of basis of $H_2(M)$, $g_{2,2}$ is divisible by $3$. 
We know that the matrix $ \begin{bmatrix}
  g_{1,1} & g_{2,1} \\ 
  g_{1,2} & g_{2,2}  
\end{bmatrix} $  is symmetric and non-singular, and by the classification of such bilinear forms \cite[Theorem 2.2]{MiHu73}, we may reduce the matrix to one of 
\[ \begin{bmatrix}
1 & 0 \\ 
 0 & 1  
\end{bmatrix},  \begin{bmatrix}
  1 & 0 \\ 
  0 & -1  
\end{bmatrix} , \mbox{ or } \begin{bmatrix}
  0 & 1 \\ 
  1 & 0  
\end{bmatrix}.  \]
In the latter two cases, we may arrange for a basis such that $g_{2,2}$ is divisible by $3$ (by changing $\alpha_2$ to $\alpha_1 + \alpha_2$ in the second case).  In the situation of the first matrix, 
we note that there is a map $S^n \to M_2$ whose mapping cone has cohomology $\Z[x]/(x^3)$ with $|x|=n$. Then, $n$ must be one of $2$, $4$, or $8$. For the generator $\iota_2$ of $\pi_2 S^2$, we have $\big[ [\iota_2, \iota_2],\iota_2\big]=0$, and for $n=4$ or $8$, we have $3\in T_n$. Thus, we always have the required fibration when $k=2$. 
\end{exam}
\vspace{0.5cm}

The case of Poincar\'{e} duality complexes $M_k$, when $k\geq 3$, is implied by the following lemma. 
\begin{lemma}\label{quadform3}
a) Let $\langle -, - \rangle$ be a symmetric bilinear form over $\Z$ of rank $\geq 3$. Then, there is a primitive $v \neq 0$ such that $\langle v,v\rangle$ is divisible by $3$. \\
b) Let $\langle -, - \rangle$ be a symmetric bilinear form over $\Z$ of rank $\geq 5$. Then, there is a primitive $v \neq 0$ such that $\langle v,v\rangle$ is divisible by $8$. \\
\end{lemma}

\begin{proof}
This argument basically follows from \cite[Ch.II, (3.2)-(3.4)]{MiHu73}. We only demonstrate how the prime $3$ argument translates to this case. Diagonalize the form over the field $\F_3$ to one which possesses diagonal entries $d_1,\cdots, d_k$ with $d_i = \pm 1$ or $0$. If all the $d_i$ are $\pm 1$, clearly 
\[ \sum_{i=1}^k d_i x_i^2 = 0 \]
has a non-zero solution $\underline{x}$ if $k\geq 3$. There exists a non-zero primitive $v$ such that $v$ reduces to $\underline{x}$ modulo $3$.  
\end{proof}

Example \ref{sphk=2} and Lemma \ref{quadform3} allow us to choose $\alpha_k$ such that $g_{k,k}$ is divisible by $3$. Thus, we conclude the proof of the following Theorem. 
\begin{theorem}\label{sphfibloc}
For $k\geq 2$, let $M_k\in \PP \DD_{n-1}^{2n}$  with $n$ even and $H_n(M;\Z) \cong \Z^k$. After inverting the primes in $T_n$, there is a  fibration $S^{n-1}\to E\to  M$ such that $E$ is homotopy equivalent to $\#^{k-1} (S^{n}\times S^{2n-1})$. 
\end{theorem}

\begin{mysubsection}{A spectral sequence argument}
We recall a well known fact that all the classes in $H_\ast E$ are transgressive in the (homology) Serre spectral sequence for $\Omega E \to PE \to E$ if the space $E$ is a suspension. 
\begin{lemma}\label{susspecseq}
    Suppose $E\simeq \Sigma X$. Then, in the Serre spectral sequence for the path$\operatorname{-}$space fibration $\Omega E \rightarrow PE \rightarrow E$, all elements  are transgressive. 
\end{lemma}

\begin{proof}
This follows from the fact that the elements in the image of the homology suspension $H_\ast \Sigma \Omega E \to H_\ast E$ are transgressive, and that $\Sigma X$ is a retract of $\Sigma \Omega \Sigma X$. 
\end{proof}

We now note down the set up in which the spectral sequence argument is carried out. We are working with a localized category of spaces in which the primes in $T_n$ are inverted, and the homology is computed with $R_n$-coefficients. The space $M_k\in \PP \DD_{n-1}^{2n}$  satisfies $H_n(M_k)\cong \Z^k$. We assume that $E_k\simeq \#^{k-1} (S^n \times S^{2n-1})$, and that $f : E_k\to M_k$ is a map which satisfies 
\begin{enumerate}
 \item $f_{*}\colon H_{n}(E_k)\to H_{n}(M_k)$ is injective, and $H_{n}M_k\cong f_\ast (H_{n}E_k)\oplus R_n\{\lambda_{k}\}$ for some $\lambda_{k}\in H_{n}M_k.$
 \item The map $ \pi_{2n-1}E_k \stackrel{f_\ast}{\to} \pi_{2n-1}M_k \stackrel{\rho}{\to} H_{2n-2}(\Omega M_k)$ induces an isomorphism  onto the quotient
\[H_{2n-2}(\Omega M_k)/\Big(\Im\big( \pi_{n}E_k \stackrel{\rho\circ f_\ast}{\to} H_{n-1}(\Omega M_k)  \big) \cdot H_{n-1}(\Omega M_k)\Big) ,\]
 where the product stands for the Pontrjagin product of $H_\ast \Omega M_k$. 
\end{enumerate}

\begin{prop} \label{spseqarg}
   With notations as above, let $\Fib(f)$ be the homotopy fibre of the map $f$. Then, $H_\ast \big(\Fib(f)\big) \cong H_\ast \big(S^{n-1}\big)$. 
\end{prop} 
 \begin{proof}
We compute  the homology Serre spectral sequence for the fibration $\Omega M_k \to \Fib(f) \to E_k$  whose $E^{2}$-page is given by 
\begin{myeq}\label{mainspseq}
E_{p,q}^{2}= H_{p}(E_k)\otimes H_{q}(\Omega M_k)\Rightarrow H_{p+q}(\Fib(f)).
\end{myeq}
We may also continue the fibration sequence further to obtain the fibration $\Omega E_k \to \Omega M_k \to \Fib(f)$. The homology of $\Omega E_k$ is described in \eqref{loophomsphprod} and the homology of $\Omega M_k$ is described in \eqref{loophomM}. These are given by
\[   H_\ast(\Omega M_k) \cong T(a_1, \cdots, a_k)/\big(l(M_k)\big), ~ H_\ast \Omega E_k \cong T(\tilde{\mu_1}, \tilde{\delta_1}, \cdots, \tilde{\mu}_{k-1}, \tilde{\delta}_{k-1})/\Big(\sum_{i=1}^{k-1} [\tilde{\mu_i}, \tilde{\delta_i}]\Big). \] 
Note that \cite[Theorem 2.8]{BaBa18} implies that both $H_\ast(\Omega M_k)$ and $H_\ast(\Omega E_k)$ are torsion-free. Let $L(a_1,\cdots, a_k)$ be the free Lie algebra on $a_1,\cdots , a_k$. We also note from \cite[Proposition 2.11]{BaBa18} that the universal enveloping algebras are computed as  
\[U\big(L(a_1,\cdots,a_k)/(l(M_k))\big) \cong  T(a_1, \cdots, a_k)/\big(l(M_k)\big),\]
\[ U\Big(L(\tilde{\mu_1}, \tilde{\delta_1}, \cdots, \tilde{\mu}_{k-1}, \tilde{\delta}_{k-1})/\big(\sum_{i=1}^{k-1} [\tilde{\mu_i}, \tilde{\delta_i}]\big)\Big) \cong T(\tilde{\mu_1}, \tilde{\delta_1}, \cdots, \tilde{\mu}_{k-1}, \tilde{\delta}_{k-1})/\big(\sum_{i=1}^{k-1} [\tilde{\mu_i}, \tilde{\delta_i}]\big).\]
Denote by $\iota$ the map from a lie algebra $L$ to it's universal enveloping algebra $U(L)$. We now apply Proposition \ref{alglem2} writing $W=(\Omega f)_\ast(H_{n-1}(\Omega E_k) ) \subset H_{n-1}(\Omega M_k) = V$. Note the commutative diagram below of graded Lie algebras and their universal enveloping algebras.
\[
\xymatrix{L(\tilde{\mu_1}, \tilde{\delta_1}, \cdots, \tilde{\mu}_{k-1}, \tilde{\delta}_{k-1})/\big(\sum_{i=1}^{k-1} [\tilde{\mu_i}, \tilde{\delta_i}]\big)  \ar[d]^{\cong}_{f_\ast} \ar[rr]^{\iota} && H_\ast(\Omega E_k) \ar[dd]^{(\Omega f)_\ast} \\ 
L^W(V,l(M_k)) \ar[r]^{\subset} \ar[d]^{\iota} & L(V,l(M_k)) \ar[d]^{\iota}  \\
U(L^W(V,l(M_k)) \ar[r]                                & U(L(V,l(M_k))) \ar[r]^{\cong} & H_\ast(\Omega M_k) }
\]
The top left vertical arrow is an isomorphism by Proposition \ref{alglem2}. This diagram allows us to identify $H_\ast(\Omega E_k)\to H_\ast(\Omega M_k)$ as $U(\hat{f})$ where $\hat{f}$ is the inclusion 
\[L^W(V,l(M_k)) \to L(V,l(M_k)) .\]
By the Poincar\'{e}-Birkhoff-Witt theorem for graded Lie algebras stated in Theorem \ref{PBW}, we have that $(\Omega f)_\ast : H_\ast \Omega E_k \to H_\ast \Omega M_k$ is injective, and in each degree, it is a torsion-free summand of a torsion-free Abelian group. Now the universal coefficient theorem implies that $(\Omega f)^\ast : H^\ast(\Omega M_k) \to H^\ast(\Omega E_k)$ is surjective. The Leray-Hirsch theorem now implies that the cohomology spectral sequence for the fibration $\Omega E_k \to \Omega M_k \to \Fib(f)$ degenerates at the second page. The same result now holds for the homology spectral sequence. As a consequence, we have that the map $H_\ast \Omega M_k \to H_\ast \Fib(f)$ is surjective. 

Now we turn our attention to the spectral sequence for the fibration $\Omega M_k \to \Fib(f) \to E_k$ \eqref{mainspseq}. As the map $H_\ast (\Omega M_k) \to H_\ast \Fib(f)$ is surjective, the $E^\infty$-page is concentrated in the $0^{th}$-column. We now calculate all the differentials that hit the $0^{th}$-column and compute the relevant cokernels. 

Consider the commutative diagram 
\[\xymatrix{ PE_k \ar[r]^{P(f)} \ar[d] & PM_k \ar[d] \\
 E_k \ar[r]^{f} & M_k. }\] 
which implies the map $PE_k \to \Fib(f)$, as $\Fib(f)$ is the homotopy pullback of $E_k \to M_k \leftarrow PM_k$. The differentials are computed via the following commutative diagram of fibrations. 
 $$\xymatrix{
\Omega \overline{E_k} \ar[r] \ar[d] & \Omega E_k \ar[d] \ar[r]^{\Omega f} & \Omega M_k  \ar[d]   \\
P\overline{E_k} \ar[r] \ar[d]  &  P E_k \ar[r] \ar[d]  & \Fib(f)  \ar[d] \\
\overline{E_k} \ar[r]  & E_k \ar[r]^{f} & M_k.}$$ 
The space $\overline{E_k}$ is the $(3n-2)$-skeleton of $E_k$ which is homotopy equivalent to $(S^n)^{\vee k-1} \vee (S^{2n-1})^{\vee k-1}$. In  the Serre spectral sequence 
$$E_{p,q}^{2}= H_{p}(\overline{E_k})\otimes H_{q}(\Omega \overline{E_k})\Rightarrow H_{p+q}(P\overline{E_k}),$$ 
the classes $\mu_i$ are transgressive by Lemma \ref{susspecseq}. It follows that in the spectral sequence 
$$E_{p,q}^{2}= H_{p}(E_k)\otimes H_{q}(\Omega E_k)\Rightarrow H_{p+q}(PE_k),$$
the classes $\mu_i$ transgress to $\tilde{\mu_i}$ and the classes $\delta_i$ transgress to $\tilde{\delta_i}$. Therefore, in the spectral sequence \eqref{mainspseq}, we have the formulas 
\[ d^n(\mu_i)= \Omega f_\ast(\tilde{\mu_i}), ~ d^{2n-1}(\delta_i)= \Omega f_\ast (\tilde{\delta_i}) ,\]
and the remaining differentials on the classes $\mu_i$ and $\delta_i$ equals $0$. 
Furthermore, from the commutative diagram 
\[ \xymatrix{ \Omega M_k \ar[r]^{=} \ar[d] & \Omega M_k \ar[d] \\ 
\Fib(f) \ar[r] \ar[d] & P(M_k) \ar[d]\\
E_k \ar[r]^f & M_k,}\]
we see that $\Omega M_k \to \Fib(f) \to E_k$ is a principal fibration. It follows that the differential in the spectral sequence \eqref{mainspseq} respects the graded (right) $H_\ast (\Omega M_k)$-module structure by a result of Moore \cite{Moo56}. More precisely, we have 
$$d^{k}(\alpha\otimes hg )=\pm d^{k}(\alpha\otimes h ) (1\otimes g),$$ 
for $\alpha \in H_{*}(E_k)$ and $g,h \in H_{*}(\Omega M_k)$. Therefore, 
\[ \Im(d^n) + \Im(d^{2n-1}) \subset E^{2}_{0,\ast} \]
equals $\Omega f_\ast\Big(H_{n-1}(\Omega E_k)\Big) \cdot H_\ast \Omega M_{k} $ $+$ $\Omega f_\ast \Big(H_{2n-2}(\Omega E_k)\Big) \cdot H_\ast (\Omega M_k)$.   This equals 
\[ W+W\cdot H_\ast(\Omega M_k) +  \Omega f_\ast \Big(H_{2n-2}(\Omega E_k)\Big) \cdot H_\ast (\Omega M_k).\]
The hypothesis 2) stated before the proposition implies that 
\[W\cdot H_{n-1}(\Omega M_k) + \Omega f_\ast \Big(H_{2n-2}(\Omega E_k)\Big) = H_{2n-2}(\Omega M_k).\]
Therefore, we have 
\[E^2_{0,\ast} /\Big( \Im(d^n) + \Im(d^{2n-1})\Big) \cong R_n\{\lambda_k\}.\]
Further, note that $\lambda_k$ cannot be hit by any differential in the spectral sequence \eqref{mainspseq} other than the transgression. It follows that $E^{\infty}_{0,\ast} \cong R_n\{\lambda_k\}$, and as we have earlier seen that this is the only possible non-zero part of the $E^{\infty}$-page, the result is proved.
\end{proof}
 
\end{mysubsection}

\section{Sphere fibrations in low dimensional cases}\label{lowdim}
In this section, we inspect the conclusions of Theorem \ref{sphfibloc} in low dimensions, specifically when $n\leq 8$. The first case is when $n=2$, where the complexes in $\PP \DD_1^4$ are simply connected $4$-manifolds. Throughout this section we use the Jacobi identity for Whitehead products \cite[Cor 7.13]{Whi78}
\begin{myeq}\label{Jacwh}
\begin{aligned}
(-1)^{pr}\big[[f,g],h\big] + (-1)^{pq}\big[[g,h],f\big] + (-1)^{rq}\big[[h,f],g\big]=0, \\
 \mbox{ for } f\in \pi_{p}(X),~g\in \pi_q(X),~h\in \pi_r(X). 
\end{aligned}
\end{myeq}
Applying \eqref{Jacwh} together with the skew symmetry for Whitehead products ($[f,g]=(-1)^{pq}[g,f]$ for $f\in \pi_p(X)$, $g\in \pi_q(X)$ ), we obtain
\begin{myeq}\label{jacalpha}
\big[[\alpha_i,\alpha_j],\alpha_k\big]= - \big[\alpha_i,[\alpha_j,\alpha_k]\big] - \big[\alpha_j,[\alpha_i, \alpha_k]\big].
\end{myeq}
\begin{mysubsection}{ Simply connected four manifolds} 
The simply connected $4$-manifolds $M_k$ for $k\geq 2$ (which are defined by $H_2(M_k) \cong \Z^k$), support a principal $S^1$-bundle whose total space is $E_k\simeq \#^{k-1}(S^2 \times S^3)$\cite{BaBa15, DuLi05}. The proof of this result relies on Smale's classification of spin $5$-manifolds. Theorem \ref{sphfibloc} provides a homotopy theoretic method to approach the situation. We first point out the argument in a simple case. 
\begin{exam}\label{connsum}
Let $M=\C P^2 \# \overline{\C P^2}$, which implies that $M$ is the mapping cone of 
$S^3 \stackrel{\eta_1 - \eta_2}{\longrightarrow} S^2 \vee S^2.$
Note that $\pi_3 S^2 \cong \Z$ implies that $T_2 = \{ 2 \}$. The argument in Example \ref{sphk=2} now implies  that after inverting $2$, we have a fibration $S^1 \to S^2 \times S^3 \to M$. This is two steps away from the geometric argument which implies that inverting $2$ is not essential, and that the fibration is a principal $S^1$-fibration. Tracing out the formula in Example \ref{sphk=2} and the fact that the triple Whitehead product of the identity map of $S^2$ is $0$, we are supposed to consider the map 
\[ S^2 \vee S^3 \to M, \mbox{ sending }  \mu_1 \mapsto \alpha_1, \delta_1 \mapsto [\alpha_1, \alpha_2].\] 
We readily compute (here $\eta_{(3)}$ is the suspension of the Hopf map which satisfies $2\eta_{(3)}=0$)
\[
\begin{aligned} 
[\mu_1, \delta_1] \mapsto & \big[ \alpha_1,[\alpha_1, \alpha_2]\big]  \\
                                       &=  [\alpha_1, \alpha_2]\circ \eta_{(3)} -[\eta_1,\alpha_2] \mbox{ by  \cite[Theorem 6.1]{Hil54}  } \\ 
                                      &=   [\alpha_1, \alpha_2]\circ \eta_{(3)} -[\eta_1-\eta_2, \alpha_2] -[\eta_2,\alpha_2]\\ 
                                     &=[\alpha_1, \alpha_2]\circ \eta_{(3)} \mbox{  as $[\eta_2, \alpha_2]=0$ }\\
                                     & \neq 0.
\end{aligned} \]
However, one may easily check that a slight tinkering of the formula : 
\[ \mu_1 \mapsto \alpha_1 + \alpha_2, \delta_1 \mapsto  \eta_2\]
does indeed yield (using the fact that $[\eta_1,\alpha_2]= [\eta_1-\eta_2,\alpha_2]=0$ and analogous formulas)
\[
\begin{aligned} 
[\mu_1, \delta_1] \mapsto &  [\alpha_1+\alpha_2,\eta_2]  \\
                              &=  [\alpha_1,\eta_2] \mbox{ as $[\alpha_2,\eta_2]=0$}\\
&=0.
\end{aligned} \]
Once we obtain the map $S^2\times S^3 \to M$, the spectral sequence argument of Proposition \ref{spseqarg} implies that the homotopy fibre is $S^1$. In order to prove that this is indeed a principal fibration, note that a $w \in H^2(M)$ which satisfies $\langle w, \alpha_1+\alpha_2\rangle=0$, is represented by a map $M\to \C P^\infty$ such that the composite $S^2\times S^3 \to M \to \C P^\infty$ is null. It follows that $S^2 \times S^3$ maps to the homotopy fibre of $w:M\to \C P^\infty$ which is easily deduced to be an equivalence. Hence, $S^1 \to S^2 \times S^3 \to M$ becomes a principal fibration.
\end{exam}

\vspace{0.5 cm}

In the general case we achieve the result by an analogous computation to Example \ref{connsum}. We are required to choose the basis $\alpha_1,\cdots, \alpha_{k-1}$ judiciously so that the formulas in Theorem \ref{sphfibloc} yield an integral result. This kind of choice was also made in \cite{BaBa15, DuLi05}. 

\begin{theorem}\label{sc4mfld}
Let $M_k$ be a simply connected $4$-manifold with $H_{2}M_k\cong \Z^k$. Then, there is a principal  $S^{1}$-fibration $S^{1}\to E_k \to M_k$ where $E_k\simeq \#^{k-1}(S^{2}\times S^{3}).$ 
\end{theorem}
\begin{proof}
Let $w\in H^2(M_k)$ be a class that  reduces to $w_2(M_k)$ modulo $2$, and choose  $\alpha_1, \cdots, \alpha_{k-1}$ to be a linearly independent set spanning a summand of  the kernel of the linear map $H_2(M_k)\to \Z$ induced by $w$. In the notation of \eqref{poM}, the matrix $\big( ( g_{i,j})\big)$ is the inverse of the matrix of the intersection form, and so by this choice $g_{i,i}$ is even for $i\leq k-1$. 
This implies $g_{i,i}[\eta_i,\alpha_j]=-g_{i,i}\big[\alpha_i,[\alpha_i, \alpha_j]\big]$ by \cite[Theorem 6.1]{Hil54}.
We define the $\beta$-classes as (where $\eta_j= \alpha_j\circ \eta$)
\begin{myeq} \label{beta4} 
\beta_{i} = \sum_{j=1}^{k-1} g_{i,j}[\alpha_{j}, \alpha_{k}] + \sum_{j=1}^{i-1} g_{i,j}[\alpha_i, \alpha_j] + \sum_{j=i+1}^k g_{i,j}\eta_j .
\end{myeq}
Then, we have (noting that $L(M_k)=\sum_{1\leq i < j \leq k} g_{i,j}[\alpha_i,\alpha_j]+  \sum_{i=1}^k g_{i,i} \eta_i$ \eqref{LMform} and the formula \eqref{jacalpha}) 
\[
\begin{aligned} 
\sum_{i=1}^{k-1}[\alpha_i, \beta_i] &= \sum_{i=1}^{k-1}\sum_{j=1}^{k-1} g_{i,j}\big[\alpha_i,[\alpha_{j}, \alpha_{k}]\big] +  \sum_{1\leq j<i\leq k-1} g_{i,j}\big[ \alpha_i,[\alpha_i, \alpha_j]\big] + \sum_{1\leq i < j \leq k} g_{i,j}[\alpha_i, \eta_j]   \\
                              &= \sum_{1\leq i < j \leq k-1} - g_{i,j} \big[[\alpha_i, \alpha_j],\alpha_k\big] - \sum_{i=1}^{k-1} g_{i,i}[\eta_i, \alpha_k]  +  \sum_{1\leq j<i\leq k-1} g_{i,j}\big[ \alpha_i,[\alpha_i, \alpha_j]\big] \\            
& - \sum_{1\leq i < j \leq k} g_{i,j}\big[[\alpha_i,\alpha_j], \alpha_j] + \sum_{1\leq i < j \leq k} g_{i,j} [\alpha_i,\alpha_j]\circ \eta_{(3)}  \mbox{ by  \cite[Theorem 6.1]{Hil54}  } \\ 
                                      &=  -  [L(M_k),\alpha_k] + L(M_k)\circ \eta_{(3)} \mbox{ as $\big[ [\alpha_k,\alpha_k],\alpha_k\big]=0$}.
\end{aligned} \]

Let $E_k=\#^{k-1}(S^2\times S^3)$. Proceeding as in the proof of Theorem \ref{sphfibloc3}, we obtain a map $f: E_k \to M_k$ sending $\mu_i$ to $\alpha_i$ and $\delta_i$ to $\beta_i$. One directly observes that the image of the $\beta_i$ in 
\[H_{2n-2}(\Omega M_k)/\Big(\Z\{a_1,\cdots, a_{k-1}\} \cdot H_{n-1}(\Omega M_k)\Big) \]
equals $-b_i$ (of Theorem \ref{sphfibloc3}). The spectral sequence argument of Proposition \ref{spseqarg} now applies to yield that the homotopy fibre of $f$ is $S^1$. The argument in the last paragraph of Example \ref{connsum} may now be repeated to deduce that this is a principal fibration. 
\end{proof}

\begin{rmk} \label{simplerbeta}
Note that the formula for the $\beta_i$ in \eqref{beta4} may be simplified in the case where $w_2(M_k)\neq 0$. Here, by choosing $\alpha_k$ such that $\pmod{2}$, $\alpha_k$ is the Poincar\'{e} dual of $w_2(M_k)$, we have from  \cite[Lemma 2.4]{CoHa90} that $g_{i,i}\equiv g_{i,k} \pmod{2}$ for $i\leq k-1$. Then, the  formula 
\begin{myeq} \label{betasimple} 
\beta_{i} = \sum_{j=1}^{k-1} g_{i,j}[\alpha_{j}, \alpha_{k}] +  g_{i,k}\eta_k 
\end{myeq}
gives
\[ \sum_{i=1}^{k-1} [\alpha_i, \beta_i]=- [L(M_k),\alpha_k].\]
\end{rmk}

\end{mysubsection}

\begin{mysubsection}{$3$-connected $8$-manifolds} As in the case of simply connected $4$-manifolds, we search for integral versions of the Theorem \ref{sphfibloc}. For this we require some results about the Whitehead products in the homotopy groups of $S^4$ \cite{Tod62, Hil54}. Recall that  $\iota_4$ is the homotopy class of the identity map $S^4 \to S^4$. 
\begin{myeq}\label{htp7-4} 
\pi_7S^4 \cong \Z\{\nu\}\oplus \Z/(12)\{\nu'\},~~   [\iota_4,\iota_4]=2\nu + \nu',
\end{myeq}
where $\nu=\mbox{ Hopf construction on the quaternionic multiplication}$. We also have using \cite[\S 4]{Hil54} and \cite[(3.7)]{Jam57}
\begin{myeq}\label{htp10-4} 
\begin{aligned}
\pi_{10}S^4 \cong \pi_{10}(S^7)\oplus \pi_9(S^3)=\Z/(24)\{x\}\oplus \Z/3\{y\}, \\
 x= \nu\circ \nu_{(7)}, ~ y= \nu'\circ \nu_{(7)}= \nu' \circ \nu'_{(7)}, ~ \nu'_{(7)}=-2\nu_{(7)}, \\
 [\iota_4,\iota_4]\circ \nu_{(7)}=2x+y,~[\nu',\iota_4]=-4x + y, \\
 [\nu,\iota_4]=2x,~\big[[\iota_4,\iota_4],\iota_4\big]=y.
\end{aligned}
\end{myeq}

We start with an example.

\begin{exam}\label{connsumhp}
Let $M=\Hyp P^2 \# \overline{\Hyp P^2}$, which is the mapping cone of $S^7 \stackrel{\nu_1-\nu_2}{\to} S^4 \vee S^4$. We let $\alpha_1$ and $\alpha_2$ denote the two wedge summands of $S^4 \vee S^4$. Form a map $E_2 \to M$ ($E_2= S^4\times S^7$) by
\[ \mu_1 \mapsto \alpha_1-\alpha_2, \mbox{ and } \delta_1 \mapsto \nu_2. \]
Observe that 
\[
\begin{aligned} 
[\mu_1, \delta_1] &\mapsto [\alpha_1 - \alpha_2, \nu_2]   \\
                         &= [\alpha_1,\nu_2] - [\alpha_2,\nu_2]  \\
                         &= 2x_1 - [L(M), \alpha_1] - 2x_2 \\
                         &=- L(M)\circ \nu'_{(7)} - [L(M),\alpha_1],
\end{aligned}
\]
so that we get a map $E_2 \to M$. Now it is easily checked that this satisfies the hypothesis of Proposition \ref{spseqarg}, and therefore, we deduce that the homotopy fibre of the map is $S^3$. We have thus constructed a fibration $S^3 \to S^4 \times S^7 \to \Hyp P^2 \# \overline{\Hyp P^2}$. We also note that this is a principal fibration as it is given by the pullback of the map $\Hyp P^2 \# \overline{\Hyp P^2} \to \Hyp P^2$ via the map which on $\pi_4$ sends both $\alpha_1$ and $\alpha_2$ to the generator. 
\end{exam}

The computations get more involved once we allow factors of $\nu'$ in the expression of $L(M)$. The following example deals with the rank $2$ case for an even intersection form, which we may assume to be the hyperbolic form by the classification \cite[Ch II (2.2)]{MiHu73}. 
\begin{exam}\label{even2}
Consider $M_2 \in \PP\DD_3^8$ determined by 
\[L(M_2)= [\alpha_1,\alpha_2]+l_1 \nu'_1 + l_2 \nu'_2.\]
Note that if $l_1$ and $l_2$ are both $1$, we cannot obtain a map $M_2 \to \Hyp P^\infty$ which on $\pi_4$ sends $\alpha_1\mapsto n_1 \iota$ and $\alpha_2 \mapsto n_2\iota$ with $\gcd(n_1,n_2)=1$. This is because $L(M_2) \mapsto n_1n_2\nu + (n_1n_2+n_1+n_2)\nu'$, and for coprime $n_1$, $n_2$, $n_1n_2+n_1+n_2$ must be odd. It follows that there is no principal $S^3$-fibration over $M_2$ in which the total space is $3$-connected. However, we can verify that for all possible values of $l_1$ and $l_2$, there  exist fibrations $S^3\to S^4 \times S^7 \to M_2$.

For constructing the fibrations, we note  it suffices to find $\mu_1$, $\delta_1$ such that $[\mu_1,\delta_1]=0\in \pi_{10}(M_2)$, and the image of  $\delta_1$ in the homology of the loop space satisfies the hypothesis of Proposition \ref{spseqarg}. We also note the symmetry between the factors $\alpha_1$ and $\alpha_2$, and also by replacing both by it's negative that $l_i\mapsto -l_i$ for $i=1,2$. Modulo these symmetries, the following formulas for $\mu_1$ and $\delta_1$ satisfy the required criteria. 
\[l_1\equiv 0 ~(\bmod{6}), l_2\equiv 0 ~(\bmod{3})\mbox{ } : \mbox{ } \mu_1=\alpha_2, \delta_1= \nu_1.\]
\[l_1\equiv 3 ~(\bmod{6}), l_2\equiv 0 ~(\bmod{3})\mbox{ } : \mbox{ } \mu_1=6\alpha_1+\alpha_2, \delta_1= \nu_1.\]
\[l_1\equiv 0 ~(\bmod{6}), l_2\equiv 2 ~(\bmod{3})\mbox{ } : \mbox{ } \mu_1=4\alpha_1+\alpha_2, \delta_1= 63\nu_1+4\nu_2.\]
\[l_1\equiv 3 ~(\bmod{6}), l_2\equiv 2 ~(\bmod{3})\mbox{ } : \mbox{ } \mu_1=10\alpha_1+\alpha_2, \delta_1= 399\nu_1+4\nu_2.\]
\[l_1\equiv 1 ~(\bmod{6}), l_2\equiv 1 ~(\bmod{3})\mbox{ } : \mbox{ } \mu_1=2\alpha_1+\alpha_2, \delta_1= 175\nu_1+44\nu_2.\]
\[l_1\equiv 4 ~(\bmod{6}), l_2\equiv 4 ~(\bmod{6})\mbox{ } : \mbox{ } \mu_1=8\alpha_1+\alpha_2, \delta_1= 257\nu_1+4\nu_2.\]
\[l_1\equiv 1 ~(\bmod{6}), l_2\equiv 2 ~(\bmod{3})\mbox{ } : \mbox{ } \mu_1=10\alpha_1+\alpha_2, \delta_1= 401\nu_1+4\nu_2-\nu_1'.\]
\[l_1\equiv 4 ~(\bmod{6}), l_2\equiv 2 ~(\bmod{3})\mbox{ } : \mbox{ } \mu_1=4\alpha_1+\alpha_2, \delta_1= 127\nu_1+8\nu_2+\nu_1'.\]
\end{exam}

\vspace{0.5cm}

We first construct the fibration when the intersection form is not even, so that  an appropriate analogue of the formula in Remark \ref{simplerbeta} works. Let $M_k \in \PP \DD_3^8$  with $H_4(M_k)=\Z^k$. Assume that $w_4(M_k)\neq 0$ which is equivalent to the intersection form not being even. Choose a basis $\alpha_1,\cdots, \alpha_k$ of $H_4(M_k)$ such that 
\begin{myeq}\label{cond4odd}
\langle \alpha_i, \alpha_k \rangle = \langle \alpha_i, \alpha_i\rangle \pmod{2} \mbox{ for } 1\leq i \leq k-1,
\end{myeq}
where $\langle -, - \rangle$ is the intersection form. This is satisfied if $ \alpha_k$ is the Poincar\'{e} dual of $w_4(M_k) \pmod{2}$ by an analogous argument to \cite[Lemma 2.4]{CoHa90}. We write $L(M_k)$ as 
\begin{myeq} \label{4lmodd} 
L(M_k)=\sum_{1\leq i < j \leq k} g_{i,j}[\alpha_i,\alpha_j] + \sum_{i=1}^k g_{i,i} \nu_i + \sum_{i=1}^k l_i \nu'_i.
\end{myeq}

\begin{prop}\label{4betaoddcase}
Suppose that $\alpha_1,\cdots, \alpha_k$ satisfy \eqref{cond4odd}, and either \\
(A) $6 \mid l_k$, or \\
(B)  $4\nmid g_{k,k}$ and $3 \mid l_k$, \\
 with $g_{k,k}$ and $l_k$  as in \eqref{4lmodd}. Under this assumption we may choose $\beta_i\in \pi_{2n-1}(M_k)$ $1\leq i \leq k-1$ such that $\sum_{i=1}^{k-1} [\alpha_i,\beta_i] = 0$, and the hypothesis of Proposition \ref{spseqarg} is satisfied.  
\end{prop} 

\begin{proof}
 Consider the formula \eqref{betasimple} and write 
\begin{myeq}\label{initbeta}
\beta'_i= \sum_{j=1}^{k-1} g_{ij}[\alpha_{j}, \alpha_{k}] +  g_{i,k}\nu_k.
\end{myeq}
We then have using \eqref{htp10-4} and  \cite[Theorem 6.1]{Hil54}
\[
\begin{aligned} 
\sum_{i=1}^{k-1}[\alpha_i, \beta'_i] &= \sum_{i=1}^{k-1}\sum_{j=1}^{k-1} g_{ij}\big[\alpha_i,[\alpha_{j}, \alpha_{k}]\big]  + \sum_{ i=1}^{k-1} g_{i,k}[\alpha_i, \nu_k]   \\
                              &= -\sum_{1\leq i < j \leq k}  g_{i,j} \big[[\alpha_i, \alpha_j],\alpha_k\big] - \sum_{i=1}^{k-1} g_{i,i}[\nu_i, \alpha_k]  +  \sum_{ i=1}^{k-1} (g_{i,i}+g_{i,k})  [\alpha_i,\alpha_k]\circ \nu_{(7)} \\ 
                                      &=  -  [L(M_k),\alpha_k] + g_{k,k}[\nu_k, \alpha_k] +\sum_{i=1}^{k} l_i [\nu'_i,\alpha_k] - \sum_{i=1}^{k-1} \Big(\frac{g_{i,i}+g_{i,k}}{2}\Big)[\alpha_i,\alpha_k]\circ \nu'_{(7)}  \\
                                      &=-[L(M_k),\alpha_k] + 2 g_{k,k} x_k - 4l_k x_k + l_ky_k + \sum_{i=1}^{k-1} l_i [\alpha_i,\nu'_k] - \sum_{i=1}^{k-1} \Big(\frac{g_{i,i}+g_{i,k}}{2}\Big)[\alpha_i,\nu'_k] .                                      
\end{aligned} \]
Note that $3\mid l_k$ implies $l_k y_k=0$. If $6 \mid l_k$, we also have $4l_k x_k =0$. Otherwise, the condition $4\nmid g_{k,k}$ implies that $4l_k \equiv 2g_{k,k} r \pmod{24}$. We now rewrite using $l_k \nu'_k\circ \nu'_{(7)}= l_k y_k=0$
\[
\begin{aligned}
(2g_{k,k}-4l_k)x_k &=- (1-r)g_{k,k} \nu_k\circ \nu'_{(7)}\\
 &= -(1-r)\Big[L(M_k)\circ \nu'_{(7)} -\sum_{1\leq i < j \leq k} g_{i,j}[\alpha_i,\alpha_j]\circ \nu'_{(7)} - \sum_{i=1}^{k-1} g_{i,i} \nu_i\circ \nu'_{(7)} -\sum_{i=1}^{k-1} l_i \nu'_i \circ \nu'_{(7)}\Big]\\
&= -(1-r)\Big[L(M_k)\circ \nu'_{(7)} - \sum_{1\leq i < j \leq k} g_{i,j}[\alpha_i,\nu'_j] + \sum_{i=1}^{k-1} g_{i,i} [\alpha_i, \nu_i] - \sum_{i=1}^{k-1} l_i \big[\alpha_i,[\alpha_i, \alpha_i]\big]\Big].
\end{aligned}
\]
Now we define $\beta_i$ by perturbing the $\beta_i'$ of \eqref{initbeta}
\[ \beta_i = \beta'_i - \big(l_i - \frac{g_{i,i}+g_{i,k}}{2}\big) \nu'_k -(1-r) \sum_{ j=i+1}^{k} g_{i,j}\nu'_j +(1-r) g_{i,i} \nu_i - (1-r)l_i [\alpha_i,\alpha_i],  \]
to get 
\[ \sum_{i=1}^{k-1} [\alpha_i,\beta_i] = -[L(M_k),\alpha_k] -(1-r) L(M_k)\circ \nu'_{(7)}.\]
We also verify easily that $\beta_i$ satisfy the requirements of Proposition \ref{spseqarg}. 
\end{proof} 

Proposition \ref{4betaoddcase} becomes applicable once we show that the hypotheses \eqref{cond4odd} and $3\mid l_k$ are always achievable. This is the subject of the following lemma. 
\begin{lemma}\label{choicebas4}
There is a choice of basis of $H_n(M_k)$ for $k\geq 2$ such that \eqref{cond4odd} holds, and either (A) or (B) of Proposition \ref{4betaoddcase} is satisfied.
\end{lemma}

\begin{proof}
We already have that if $\alpha_k$ is the Poincar\'{e} dual of $w_4(M_k) \pmod{2}$, \eqref{cond4odd} is satisfied.  We now assume that the basis $\alpha_i$ is chosen such that the induced inner product is diagonal modulo $3$. We note that the change of $l_i$ with a change of basis is not linear, however the following change of basis formulas hold by applying \eqref{htp10-4}. 
\[ (\alpha_k \mapsto - \alpha_k) \to (l_k \mapsto -l_k + g_{k,k}). \] 
\[ \begin{pmatrix}\alpha_{k-1}\mapsto \alpha_{k-1}-\alpha_k\\ \alpha_k \mapsto \alpha_k\end{pmatrix} \to (l_k \mapsto l_k+l_{k-1}+g_{k,k-1}).\]
\[\begin{pmatrix} \alpha_{k-1}\mapsto \alpha_{k-1}-2\alpha_k\\ \alpha_k \mapsto \alpha_k\end{pmatrix} \to (l_k \mapsto l_k+2l_{k-1}+ g_{k-1,k-1}+2g_{k,k-1}).\]
It is now clear that a combination of the above manoeuvres allow us to arrange for $3\mid l_k$. 

Suppose that the basis that the numbers $g_{k,k}$ and $l_k$ obtained above satisfy $4 \mid g_{k,k}$ and $2 \nmid l_k$. As the intersection form is odd, we must have $i$ such that $g_{i,i}$ is odd. The transformation 
\[\begin{pmatrix} \alpha_{i}\mapsto \alpha_{i}-6\alpha_k\\ \alpha_k \mapsto \alpha_k\end{pmatrix} \to (l_k \mapsto l_k+6l_{i}+ 15g_{i,i}+6g_{k,i}),\]
allows us to make $6 \mid l_k$.
\end{proof}

Now consider the case when the intersection form is even. Assuming $L(M_k)$ as in \eqref{4lmodd} we note that this implies that the $g_{i,i}$ are even for all $i$. We now adapt the formula \eqref{beta4} in this case to prove the following result. 

\begin{prop}\label{4betaevencase}
Suppose that $k\geq 2$, $24\mid g_{k,k}$, and $l_k=0$. Under this assumption we may choose $\beta_i\in H_{2n-1}(M_k)$ such that the hypothesis of Proposition \ref{spseqarg} is satisfied.  
\end{prop}

\begin{proof}
We define the classes
\[\beta'_{i} = \sum_{j=1}^{k-1} g_{ij}[\alpha_{j}, \alpha_{k}] + \sum_{j=1}^{i-1} g_{i,j}[\alpha_i, \alpha_j] + \sum_{j=i+1}^k g_{i,j}\nu_j - l_i \nu'_k .\]
With this choice we compute using \eqref{htp10-4} as in Proposition \ref{4betaoddcase}
\[
\begin{aligned} 
\sum_{i=1}^{k-1}[\alpha_i, \beta'_i] &= \sum_{i=1}^{k-1}\sum_{j=1}^{k-1} g_{ij}\big[\alpha_i,[\alpha_{j}, \alpha_{k}]\big] +  \sum_{1\leq j<i\leq k-1} g_{i,j}\big[ \alpha_i,[\alpha_i, \alpha_j]\big] \\ 
& + \sum_{1\leq i < j \leq k} g_{i,j}[\alpha_i, \nu_j] - \sum_{i=1}^{k-1} l_i [\alpha_i,\nu'_k]  \\
                              &= \sum_{1\leq i < j \leq k-1} - g_{i,j} \big[[\alpha_i, \alpha_j],\alpha_k\big]  - \sum_{i=1}^{k-1} g_{i,i}[\nu_i, \alpha_k]  + \sum_{i=1}^{k-1} g_{i,i} [\alpha_i, \alpha_k] \circ \nu_{(7)} - \sum_{i=1}^{k-1} l_i [ \nu'_i,\alpha_k]\\
                              & +  \sum_{1\leq j<i\leq k-1} g_{i,j}\big[ \alpha_i,[\alpha_i, \alpha_j]\big]  - \sum_{1\leq i < j \leq k} g_{i,j}\big[[\alpha_i,\alpha_j], \alpha_j] + \sum_{1\leq i < j \leq k} g_{i,j} [\alpha_i,\alpha_j]\circ \nu_{(7)} \\ 
                                      &=    -[L(M_k),\alpha_k] + g_{k,k} [\nu_k,\alpha_k] + L(M_k)\circ \nu_{(7)} + \sum_{i=1}^{k-1} g_{i,i} [\alpha_i,\alpha_k] \circ \nu_{(7)} \\
&- \sum_{i=1}^{k-1} g_{i,i}x_i - \sum_{i=1}^{k-1} l_i y_i - g_{k,k}x_k \\ 
                                      &= -[L(M_k),\alpha_k]  + L(M_k)\circ \nu_{(7)} + \sum_{i=1}^{k-1} g_{i,i} [\alpha_i,\alpha_k] \circ \nu_{(7)} - \sum_{i=1}^{k-1} g_{i,i} x_i  - \sum_{i=1}^{k-1} l_i y_i +  g_{k,k}x_k .
\end{aligned} \]
As $24\mid g_{k,k}$, $g_{k,k}x_k=0$. We write
\[
\begin{aligned}
 g_{i,i} [\alpha_i,\alpha_k] \circ \nu_{(7)} &= -\frac{g_{i,i}}{2} [\alpha_i,\alpha_k] \circ \nu'_{(7)} \\
 &= -\frac{g_{i,i}}{2} [\alpha_i,\nu'_k].
\end{aligned}\]
We now define $\beta_i$ by perturbing $\beta'_i$ as 
\[ \beta_i = \beta'_i - \frac{g_{i,i}}{2}\cdot \nu_i +  \frac{g_{i,i}}{2}\cdot \nu'_k - l_i[\alpha_i,\alpha_i], \]
and from the formulas \eqref{htp10-4}, it follows that 
\[\sum_{i=1}^{k-1} [\alpha_i,\beta_i] =- [L(M_k),\alpha_k] - L(M_k)\circ \nu_{(7)}.\]
Clearly the $\beta_i$ satisfy the hypothesis of Proposition \ref{spseqarg}. 
\end{proof}
We now summarize the computations in the following theorem. 
\begin{theorem}\label{sphfib4}
Let $M_k\in \PP \DD_3^8$, that is, $H_4(M_k)= \Z^k$ for $k\geq 2$. Such an $M_k$ supports a $S^3$-fibration $S^{3}\to E_k\to M_k$ with $E_k \simeq \#^{k-1}(S^4 \times S^7)$. 
\end{theorem}
\begin{proof}
We follow the proof of Theorem \ref{sphfibloc}. The choice of $\alpha_i$ and $\beta_i$ are made in Proposition \ref{4betaoddcase} and Lemma \ref{choicebas4} if the intersection form is not even, and in Proposition \ref{4betaevencase} if the intersection form is even. In the latter case, we need to arrange that $24\mid g_{k,k}$. For this, we use the fact that $g_{i,i}$ is even to rewrite 
\[
L(M_k)=\sum_{1\leq i < j \leq k} g_{i,j}[\alpha_i,\alpha_j] + \sum_{i=1}^k \frac{g_{i,i}}{2} [\alpha_i,\alpha_i] + \sum_{i=1}^k s_i \nu'_i.
\]
Note that the $s_i$ change linearly with $\alpha_i$ $\pmod{12}$. Now we write 
\[ s_1 \nu'_1+\cdots + s_k \nu'_k = d \tau \]
where $d$ equals the greatest common divisor of the $s_i$. Extending $\tau$ to a basis we assume that $s_i=0$ if $i\geq 2$. Consider the summand spanned by the last $k-1$ basis elements. By Lemma \ref{quadform3}, we obtain a primitive $v\neq 0$ such that $\langle v,v\rangle$ is divisible by $24$ if $k\geq 6$. Extend $\tau,v$ to a basis, to verify the required criteria.  For $k\leq 5$, we are done by the classification in \cite[Ch II (2.2)]{MiHu73}. More precisely, if $k\leq 5$, $k$ can be $2$ or $4$, and in each case the intersection form is a direct sum of copies of the hyperbolic form. If $k=2$, we are done by Example \ref{even2}. If $k=4$, we consider the $\tau$ above and note that by adding multiples of $12$ to the $s_i$, and the fact that their $\gcd$ is $1$, we may choose two $s_{i_1},s_{i_2}$ among $s_1,\cdots s_4$ that are relatively prime. Choosing $v=s_j$, for $j\not\in \{i_1,i_2\}$, allows us to proceed as before. Therefore, we have a fibration $E_k \to M_k$ for $E_k \simeq \#^{k-1}(S^4\times S^7)$ with fibre $S^3$. 
\end{proof}

\end{mysubsection}

\begin{mysubsection}{$7$-connected $16$-manifolds}
One may approach  results for $\PP\DD_7^{16}$ in an analogous manner to those for $\PP \DD_3^8$, and expect that the existence of the Hopf invariant one class $\sigma \in \pi_{15}(S^8)$ will allow us to construct integral versions of Theorem \ref{sphfibloc}. However, this is not the case. First we note some formulas for Whitehead products and compositions in the homotopy groups of $S^8$.  \cite[(5.16)]{Tod62}, \cite[(7.4)]{Mim67}
\begin{myeq}\label{htp22-8}
\begin{aligned}
\pi_{15}S^8 \cong \Z\{\sigma\}\oplus \Z/(120)\{\sigma'\},~~   [\iota_8,\iota_8]=2\sigma - \sigma',\\
\pi_{22}S^8 \cong \pi_{22}(S^{15})\oplus \pi_{21}(S^7)=\Z/(240)\{z\}\oplus \Z/(24)\{u\} \oplus \Z/4\{v\}, \\
 z= \sigma\circ \sigma_{(15)}, ~ u= \sigma'\circ \sigma_{(15)}, \\
 [\sigma,\iota_8]=2z- u \pm 8u,~\big[[\iota_8,\iota_8],\iota_8\big]=\pm 8 u.
\end{aligned}
\end{myeq}
From the formula \eqref{htp22-8}, one can easily deduce via a direct calculation that it is not possible to construct a fibration $S^7 \to S^8\times S^{15} \to \cO P^2 \# \overline{\cO P^2}$. For example, the formulas in Example \ref{connsumhp} does not generalize here, as the terms $[\iota_8,\sigma]$ contain multiples of $u$ which are not expressible in the form $\sigma \circ -$. It is possible to lay down conditions under which the formulas do yield the desired fibrations. 
 We leave the study of these integral fibrations for a future publication.

\end{mysubsection}

\begin{mysubsection}{$5$-connected $12$-manifolds}
For manifolds in $\PP \DD^{12}_5$ we know that the fibration $S^5 \to E_k \to M_k$ cannot exist over the integers as there is no Hopf invariant one class in $\pi_{11}(S^6)$. Therefore we have to invert $2$. However, $\pi_{11}S^6 = \Z$, so it is not necessary to invert anything else. 

\end{mysubsection}

\section{Sphere fibrations for manifolds with high Betti number} \label{highbetti}
In this section, we improve the results of \S \ref{loccat} in the sense that we reduce the number of primes that are needed in the localization. We show that for $k$ greater than the number of cyclic summands in $\pi_{n-1}^s\otimes \Z[\frac{1}{2}]$, the sphere fibrations exist once $2$ is inverted. Further if $3$ is inverted, $S^{n-1}$ is an $A_3$-space \cite{Ada61,Zab70}, and in this case, these fibrations are obtained as a pullback of the $S^{n-1}$-fibration over the associated projective plane \cite{Sta63}.  Throughout this section we work in the category of spaces with $2$ inverted, and write $R_2 = \Z[1/2]$ and $R_{2,3} = \Z[1/2,1/3]$. 

\begin{mysubsection}{Whitehead products in $\pi_{2n-1} S^n$} 
After inverting the prime $2$, for $n$ even, $\Omega S^n$ splits into  a product $\Omega S^{2n-1} \times S^{n-1}$. The map $S^{2n-1} \to S^n$ which induces the inclusion of $\Omega S^{2n-1}$ may be chosen to be $\frac{1}{2}[\iota_n,\iota_n]$. In these terms we have 
\begin{myeq}\label{gen2n-1}
\pi_{2n-1}(S^n)\otimes R_2 \cong R_2\{ [\iota_n,\iota_n]\} \oplus E(\pi_{2n-2}(S^{n-1})\otimes R_2).
\end{myeq} 
We now note the following formula for the Whitehead products for the generators in \eqref{gen2n-1} using the Jacobi identity and \cite[Theorem 6.1]{Hil54}
\begin{myeq}\label{wh2n-1}
3\big[\iota_n,[\iota_n,\iota_n]\big]=0,~ ~ [\iota_n, E\alpha] = [\iota_n,\iota_n]\circ \Sigma^n \alpha ~ \forall ~ \alpha\in \pi_{2n-2}S^{n-1} \otimes R_2. 
\end{myeq}
The last equation is implied by the fact that the Hopf invariant of $E(\alpha)$ is $0$. It is worthwhile to note here that $\Sigma^n \alpha \in \pi_{3n-2}(S^{2n-1})$ belongs to the stable range, so that the right hand side of the second equation of \eqref{wh2n-1} is non-zero only when $\alpha$ represents a non-trivial stable homotopy class. 

\end{mysubsection}

\begin{mysubsection}{Constructing spherical fibrations after inverting $2$}
As in \S \ref{loccat}, by obstruction theory we construct a map $E_k\simeq \#^{k-1}(S^n \times S^{2n-1})\to M_k$ for $M_k \in \PP \DD^{2n}_{n-1}$, which means $H_n(M_k)\cong \Z^k$. One should note that unless we invert $2$, the homomorphism 
\[ \pi_s(S^{n-1}) \to \pi_s(\Omega M_k) \cong \pi_{s+1}(M_k) \]
associated to a summand $S^n \to (S^n)^{\vee k} \to M_k$ has non-trivial kernel. This follows from the EHP-sequence for a sphere \cite{Jam57Ann} and the fact that the inclusion of a sphere induces a summand on the level of homotopy groups \cite{BaBa18, BeTh14}.  After inverting $2$, the kernel vanishes. In this situation, we enumerate criteria under which it is possible to construct a map $E_k\to M_k$.  Recall, the notations $\mu_i, \delta_i$ of homology generators of $E_k$ and $\alpha_i$ of $M_k$, and the attaching map $L(M_k)$ from \S \ref{alglem}. 
\begin{prop}\label{obstn}
Suppose that the attaching map $L(M_k) \in \pi_{2n-1}\big((S^n)^{\vee k}\big)$ of $M_k$ takes the form \eqref{LMform} (for an invertible integer matrix $\big( ( g_{i,j})\big)$) 
\begin{myeq}\label{LMgenform}
 L(M_k)= \sum_{1\leq i<j \leq k} g_{i,j} [\alpha_i, \alpha_j] +\sum_{i=1}^k  \big( g_{i,i} (\frac{1}{2} [\alpha_i,\alpha_i]) + \alpha_i \circ \omega_i \big), \mbox{ for } \omega_i \in E(\pi_{2n-2}(S^{n-1})).
 \end{myeq}
Assume that this satisfies \\
1. $g_{k,k} \equiv 0 \pmod{3}$. \\
2. $\omega_k$ lies in the kernel of $\Sigma : \pi_{2n-1}(S^n) \to \pi_{n-1}^s$. \\
Then, there is a map $E_k \to M_k$ which sends $\mu_i $ to $\alpha_i$, and that $\beta_i$ satisfy the conditions of Proposition \ref{spseqarg}. 
\end{prop}

\begin{proof}
We define the $\beta$-classes as in \eqref{betasimple} 
\begin{myeq} \label{betagen}
\beta_i = \sum_{j=1}^{k-1} g_{i,j} [ \alpha_j,\alpha_k] + \frac{1}{2} g_{i,k}[\alpha_k, \alpha_k] - \alpha_k \circ \omega_i. 
\end{myeq} 
Observe that the elements $\omega_i$ are in the kernel of $\rho : \pi_{2n-1}(S^n) \to H_{2n-2}(\Omega S^n)$. Thus, the $\beta_i$ of \eqref{betagen} have the same image in $H_\ast \Omega M_k$ as those of \eqref{betai}, and so it satisfies the criteria of Proposition \ref{spseqarg}. We complete the proof by noting
\[
\begin{aligned} 
\sum_{i=1}^{k-1}[\alpha_i, \beta_i] &= \sum_{i=1}^{k-1}\sum_{j=1}^{k-1} g_{i,j}\big[\alpha_i,[\alpha_{j}, \alpha_{k}]\big] +  \sum_{i=1}^{k-1} \frac{1}{2}g_{i,k}\big[ \alpha_i,[\alpha_k, \alpha_k]\big] - \sum_{ i=1}^{k-1} [\alpha_i, \alpha_k \circ \omega_i]   \\
                              &= \sum_{1\leq i < j \leq k} - g_{i,j} \big[[\alpha_i, \alpha_j],\alpha_k\big] - \sum_{i=1}^{k-1} \frac{1}{2}g_{i,i}\big[[\alpha_i,\alpha_i], \alpha_k\big]  -  \sum_{i=1}^{k-1} [\alpha_i\circ \omega_i, \alpha_k] \\            
                                      &=  -  [L(M_k),\alpha_k] + \frac{1}{2} g_{k,k}\big[\alpha_k, [\alpha_k, \alpha_k]\big] + [\alpha_k\circ \omega_k, \alpha_k] \\ 
                                       &= 0.
\end{aligned} \]
The last equality follows from \eqref{wh2n-1} and the hypothesis 1 and 2 of the proposition.
\end{proof}

Proposition \ref{obstn} lays down the conditions we need to arrange in order to construct a map $E_k \to M_k$ whose homotopy fibre is $S^{n-1}$. For a finite Abelian group $A$, we define the number of cyclic summands to be the number $r$ in it's decomposition as 
\[A \cong \Z/(a_1)\oplus \Z/(a_2) \cdots \oplus \Z/(a_r), \mbox{ with } a_i \mid a_{i+1}.\]
In this notation, we have the following theorem. 
\begin{theorem}\label{sphfibgen}
Let $M_k\in \PP \DD_{n-1}^{2n}$  with $H_nM_k\cong \Z^k$, and $n>8$. Let $r$ be the number of cyclic summands of $\pi_{n-1}^s$. If $k>r$,  after inverting $2$ there is a fibration $E_k \to M_k$ with fibre $S^{n-1}$ where $E_k\simeq \#^{(k-1)}(S^{n}\times S^{2n-1})$.
\end{theorem}

\begin{proof}
We apply Proposition \ref{spseqarg} and Proposition \ref{obstn}. We only need to show that one may choose a basis of $\pi_n(M_k)$ such that the hypotheses of Proposition \ref{obstn} are satisfied. This is done for $r=1$ in Example \ref{cyclic} and for $r>1$ in Proposition \ref{basisgr1}
\end{proof}

In the following example we work out the details  when $\pi_{n-1}^s$ is cyclic which is analogous to the methods of \S \ref{lowdim}.

\begin{exam}\label{cyclic}
Suppose that $\pi_{n-1}^s$ is cyclic with generator $\chi$ and of order $d$, and assume as in Proposition \ref{obstn}, 
\[ L(M_k)= \sum_{1\leq i<j \leq k} g_{i,j} [\alpha_i, \alpha_j] + \sum_{i=1}^k  \big( g_{i,i} (\frac{1}{2} [\alpha_i,\alpha_i]) + \alpha_i \circ \omega_i \big), \mbox{ for } \omega_i \in E(\pi_{2n-2}(S^{n-1})).\] 
We assume $k\geq 2$, and show that it is possible to change the basis $\{ \alpha_i\}$ so that the hypothesis 1 and 2 of Proposition \ref{obstn} are satisfied. We may now write (for $\Sigma : \pi_{2n-1}(S^n) \to \pi_{n-1}^s$) 
\[ \Sigma \omega_i = x_i \chi \mbox{  for } x_i \in \Z/d. \]
We take $x$ to be the greatest common divisor of the $x_i$ in $\Z/d$ and write $x_i= c_i x$. The $c_i$ may be arranged so that they have no common divisor over $\Z$. Now change the basis to $\{\alpha_1',\cdots, \alpha_k'\}$ so that the first element is $\alpha_1'=\sum c_i \alpha_i$. In this new basis, we have (for some representative $u\in \pi_{2n-1}(S^n)$ of $\chi\in \pi_{n-1}^s$)
\[ \Sigma (\alpha_1' \circ x u)-\Sigma(\sum_{i=1}^k \alpha_i \circ \omega_i ) =0 .\]
This arranges the $2^{nd}$ hypothesis whenever $k\geq 2$. We now arrange for $g_{k,k}\equiv 0 \pmod{3}$ without changing the first element. If $k\geq 4$, applying Lemma \ref{quadform3} to the summand spanned by the last $k-1$ elements yields a change of basis with $g_{k,k}\equiv 0 \pmod{3}$. It remains to consider the cases $k=2$ and $k=3$. For $k=3$, \cite[Theorem 2.2]{MiHu73} implies that there is a $1$-dimensional summand with self intersection $\pm 1$. This implies $n\leq 8$, cases which are already dealt with in \S \ref{lowdim}. 

In the case $k=2$, as $n\neq 2$, $4$, or $8$ we do not have a Hopf invariant one class, so that the matrix $\big((g_{i,j})\big)$ with respect to some basis is $\begin{pmatrix} 0 & 1 \\ 1 & 0 \end{pmatrix}$. Suppose that $c_1\alpha_1+c_2\alpha_2$ is the element constructed above used to kill off the $\omega_i$. As $c_1$ and $c_2$ don't have common divisors, one of them is not divisible by $3$, say $c_1$. Then, we may choose a second element of the basis as one which is $ \alpha_2 \pmod{3}$ using the fact that $GL_n \Z \to GL_n \F_3$ is surjective. This satisfies $g_{2,2}$ is $0 \pmod{3}$. 
\end{exam}

The general case with more than $1$ cyclic summand is a repeated iteration of the argument in Example \ref{cyclic}.

\begin{prop} \label{basisgr1}
Suppose that $M_k$ is as in Theorem \ref{sphfibgen}. Then, there is a choice of basis that satisfies the hypothesis of Proposition \ref{obstn}.
\end{prop}

\begin{proof}
The case $r=1$ is covered in Example \ref{cyclic}. Note that the technique of Example \ref{cyclic} may be applied to a cyclic summand $C$ of $\pi_{n-1}^s$ to yield a basis where the choice of $\omega_i$ in \eqref{LMgenform} satisfies $pr_C(\Sigma \omega_i)=0$ if $i>1$ with $pr_C$ standing for the projection of $\pi_{n-1}^s$ onto the summand $C$. We now apply this fact one summand at a time to obtain a basis with the property that for $i>r$, $\Sigma \omega_i =0$. We now arrange for $g_{k,k}\equiv 0 \pmod{3}$ without changing the first $r$ basis elements. 

Let $\{ \alpha_i \}$ be the basis obtained so far, and let $\langle -, - \rangle$ be the intersection form. Suppose $\langle \alpha_k, \alpha_k\rangle \not\equiv 0 \pmod{3}$. If $k\geq r+3$, then take the summand of $\Z^k$ spanned by $\alpha_{r+1}, \cdots, \alpha_k$  and apply Lemma \ref{quadform3} to get a change a basis so that $\langle \alpha_k, \alpha_k \rangle \equiv 0 \pmod{3}$. In the other cases we change $\alpha_k$ by adding a linear combination of $\alpha_1,\cdots, \alpha_{k-1}$ so that  $\langle \alpha_k, \alpha_k \rangle \equiv 0 \pmod{3}$. We work over $\F_3$ and then lift the change of basis to an integral one. It is possible to choose a basis $\{v_1,\cdots, v_{k-1}\}$  of the summand $B$ generated by $\alpha_1,\cdots, \alpha_{k-1}$ such that $\langle v_i,\alpha_k \rangle \equiv 0 \pmod{3}$ for $i\leq k-2$. If in addition $\langle v_{k-1}, \alpha_k \rangle \equiv 0 \pmod{3}$, then, the restriction of the intersection form to $B$ is non-singular over $\F_3$. So, we may diagonalize it to one with non-zero entries. 
As $r\geq 2$, $k\geq 3$, so there are at least $2$ of the $v_i$, so there is a combination $u=\alpha_k + \sum c_i v_i$ such that $\langle u, u\rangle \equiv 0 \pmod{3}$. As the $v_i$ involve only $\alpha_1, \cdots, \alpha_{k-1}$ so we may replace $\alpha_k$ with $u$ and the proof is complete. 

If $c=\langle v_{k-1},\alpha_k \rangle \not\equiv 0 \pmod{3}$, we note that for $t\not\equiv 0 \pmod{3}$,
\[ \langle \alpha_k + tv_{k-1}, \alpha_k+tv_{k-1} \rangle \equiv \langle \alpha_k,\alpha_k\rangle + \langle v_{k-1}, v_{k-1} \rangle + 2 tc \pmod{3} \]
which may be arranged to be $0$ as long as 
\begin{myeq}\label{condmod3}
\langle \alpha_k,\alpha_k\rangle + \langle v_{k-1}, v_{k-1} \rangle \not\equiv 0 \pmod{3}.
\end{myeq}
In case \eqref{condmod3} does not hold consider $v_j $ for $j<k-1$. There is at least one such $j$ as $k\geq 3$. If $\langle v_j, v_j \rangle \not \equiv 0 \pmod{3}$, then change $\alpha_k$ to $\alpha_k+v_j$ to ensure \eqref{condmod3} holds. If further $\langle v_j,v_j\rangle \equiv 0\pmod{3}$ and $\langle v_j,v_{k-1} \rangle\not\equiv 0 \pmod{3}$, then replace $v_{k-1}$ with $v_{k-1}' = v_{k-1}+sv_j$ so that $\langle v_{k-1}',v_{k-1}'\rangle \equiv 0 \pmod{3}$ and this implies that \eqref{condmod3} holds. Finally, if all the $v_j$ satisfy $\langle v_j,v_{k-1} \rangle \equiv 0 \pmod{3}$, then over $\F_3$, the form breaks up into orthogonal pieces spanned by $\{v_1,\cdots,v_{k-2}\}$ and $\{v_{k-1},\alpha_k\}$. Thus, the restriction of the intersection form to the summand spanned by $v_1,\cdots, v_{k-2}$ is non-singular over $\F_3$. There is a linear combination $v$ of these $v_j$ which satisfies $\langle v,v\rangle \not\equiv 0 \pmod{3}$. Again we may add this to $\alpha_k$ to ensure \eqref{condmod3} holds. This completes the proof. 
\end{proof}

\end{mysubsection}

\begin{mysubsection}{Sphere fibrations as pullbacks}
Now that we have  sphere fibrations $S^{n-1} \to E_k \to M_k$ for large enough $k$ after inverting $2$, we may ask when these are realizable pullbacks. More precisely, we would like to build up an analogous story to \S \ref{lowdim} where the fibrations were principal fibrations. However, we know that the spheres do not usually possess a group structure except for $S^1,S^3$ or $S^{n-1}$ after $p$-completion for $n\mid 2p-2$ as the Sullivan spheres \cite{Sul70, Wil73}. On the other hand, the odd spheres possess a homotopy associative multiplication once $2$ and $3$ are inverted \cite{Ada61, Zab70}. However, inverting $3$ is necessary here otherwise we would only have a non-homotopy associative $H$-space structure \cite{Jam57}. We briefly recall the construction using the following \cite[Corollary 3.2]{Yam89} 
\begin{prop} \label{equivsphinfloop}
After inverting $2$ and $3$, the inclusion $S^{n-1} \to Q(S^{n-1})$ is a $(5n +1)$-equivalence. 
\end{prop}

A direct corollary of Proposition \ref{equivsphinfloop} is that the $E_\infty$-space structure on $Q(S^{n-1})$ yields a homotopy associative structure on $S^{n-1}$. 

\begin{cor} \label{assmult}
After inverting the primes $2$ and $3$, $S^{n-1}$ has a homotopy associative multiplication. 
\end{cor}

\begin{proof}
Since $QS^{n-1}$ has a homotopy associative multiplication : $m : QS^{n-1}\times QS^{n-1}\to QS^{n-1}$. We define the multiplication on $S^{n-1}$ by lifting
\[\phi:  S^{n-1}\times S^{n-1}\stackrel{\iota \times \iota}{\to} QS^{n-1}\times QS^{n-1} \stackrel{m}{\to} QS^{n-1} \]
via the isomorphism $[S^{n-1} \times S^{n-1}, S^{n-1}] \cong [S^{n-1} \times S^{n-1}, QS^{n-1}]$ implied by Proposition \ref{equivsphinfloop}. Associativitiy is implied by the associativity of $m$ and the isomorphism 
\[ [S^{n-1} \times S^{n-1}\times S^{n-1}, S^{n-1}] \cong [S^{n-1} \times S^{n-1}\times S^{n-1}, QS^{n-1}].\]
\end{proof}

Following Stasheff \cite{Sta63}, we consider the projective planes $P_2(X)$ for an $A_3$-space $X$ which support a fibration $E_2(X) \to P_2(X)$ with an action $X\times E_2(X) \to E_2(X)$ identifying $X$ as the fibre. In fact we have the diagram 
\[ \xymatrix{ E_1(X) \ar[d]^{H(m_X)} \ar[r]   & E_2(X) \ar[d]\\ 
P_1(X) \ar[r] & P_2(X) } \] 
with $H(m_X)$ standing for the Hopf construction on the multiplication $m_X$ in the sequence of identifications 
\[ E_1(X) \simeq X \ast X \stackrel{H(m_X)}{\to} \Sigma X \simeq P_1(X),~ P_2(X) \simeq C(H(m_X)),~ E_2(X) \simeq X\ast X\ast X.\] 
We now fix a model for $QS^{n-1}$ as a $A_\infty$-space via the identification $Q(S^{n-1}) \simeq \Omega Q(S^n)$. Therefore, Corollary \ref{assmult} implies the following commutative diagram by the functoriality of the $P_2$-construction. 
\begin{myeq} \label{Hmapproj}
\xymatrix{ E_2(S^{n-1}) \ar[r] \ar[d] & E_2(Q(S^{n-1})) \ar[r] \ar[d] & E\Omega QS^n \simeq \ast \ar[d] \\ 
P_2(S^{n-1}) \ar[r] & P_2(Q(S^{n-1})) \ar[r] & B\Omega QS^n \simeq QS^n } 
\end{myeq}
We also note that the construction in Corollary \ref{assmult} makes the multiplication $m$ on $S^{n-1}$ homotopy commutative. This implies using \cite[Lemma 2.4]{Jam57I} that
\begin{myeq}\label{whpmultcomm}
2H(m) = - [\iota_n,\iota_n].
\end{myeq}

We now prove that the sphere fibration over $M_k$ is obtained as a pullback of $E_2(S^{n-1}) \to P_2(S^{n-1})$. 
\begin{thm}\label{sphfibpbk}
With notations as in Theorem \ref{sphfibgen}, after inverting $2$ and $3$, the fibration $E_k\to M_k$ is a homotopy pullback of $M_k \to P_2(S^{n-1})\leftarrow E_2(S^{n-1})$ for a suitable map $M_k\to P_2(S^{n-1})$.  
\end{thm}

\begin{proof}
We define the map $s_k: (S^n)^{\vee k} \to S^n$ which quotients out the first $k-1$ factors, and then compose it to $P_2(S^{n-1})$. Now consider the diagram 
\[\xymatrix{ S^{2n-1} \ar[d]^{L(M_k)} \ar[r] & \DD^{2n} \ar[d] \\ 
(S^n)^{\vee k} \ar[r] \ar[rrd]_{s_k} & M_k \ar@{-->}[rd] \\ 
& & P_2(S^{n-1}) } \]
Assume that $L(M_k)$ satisfies the form given in \eqref{LMgenform}. In this form, applying Proposition \ref{equivsphinfloop}, we see that $\Sigma \omega_k=0$ implies that $\omega_k=0$. We now observe via \eqref{whpmultcomm} that
\[ s_k \circ L(M_k)= g_{k,k} H(m). \]
As $P_2(S^{n-1})$ is the mapping cone of $H(m)$, the dashed arrow exists in the above diagram, and hence we  obtain a map $\phi_k: M_k\to P_2(S^{n-1})$. Let $E=\phi_k^\ast(E_2(S^{n-1}))$ be the homotopy pullback $M_k\to P_2(S^{n-1}) \leftarrow E_2(S^{n-1})$. We shall show that $E\simeq E_k$ by lifting the map $E_k \to M_k$ to $E$.  Once we are able to do this, it follows that the map $E_k\to E$ is a homology equivalence proving the required result.

Consider the following diagram 
\begin{myeq} \label{pullbackdiag}
\xymatrix{S^{3n-2} \ar[d]^\chi & S^{n-1} \ar[r]^{=} \ar[d] & S^{n-1} \ar[r] \ar[d] & QS^{n-1} \ar[r] \ar[d] & QS^{n-1} \ar[d] \\
X \ar[d] & E \ar[r] \ar[d] & E_2(S^{n-1}) \ar[r] \ar[d] & W \ar[r] \ar[d] & P QS^{n} \ar[d] \\
E_k \ar[r] & M_k \ar[r] & P_2(S^{n-1}) \ar[r]^{=} & P_2(S^{n-1}) \ar[r] & Q S^{n}. }
\end{myeq}
In \eqref{pullbackdiag}, $W$ is defined so that the right-most square is a homotopy pullback. Note that Proposition \ref{equivsphinfloop} implies that $E_2(S^{n-1})\to W$ is a $(5n+1)$-equivalence. Now the $(2n-1)$-skeleton $X$ of $E_k$ is $(S^n)^{\vee k-1} \vee (S^{2n-1})^{\vee k-1}$, and from the formula of $\alpha_i$ and $\beta_i$ we check that the composite $X \to M_k \to P_2(S^{n-1})$ is null-homotopic. Therefore, the map of $X$ all the way to $QS^{n}$ lifts to $PQS^{n}$. Hence, it lifts to $W$ being the pullback, and to $E_2(S^{n-1})$ from the connectivity of the map $E_2(S^{n-1})\to W$. Again as $E$ is the pullback, we get a lift of $X \to E$. Let $\chi$ stand for the attaching map of the $(3n-1)$-cell of $E_k$. The composite 
\[S^{3n-2} \stackrel{\chi}{\to} X \to E \to   M_k \]
is trivial, so the composite to $E$ lifts to $S^{n-1}$. This implies that the composite to $E$ is trivial as the inclusion $S^{n-1} \to E$ is null-homotopic. The last statement comes from the fact that the map $\Omega M_k \to S^{n-1}$ is surjective on $\pi_{n-1}$. Therefore, there is a lift $E_k \to E$. 
\end{proof}

 \end{mysubsection}

\section{Applications for loop space decompositions} \label{appln}

In this section, we use the spherical fibrations of \S \ref{highbetti} to deduce new results for loop space decompositions. The spherical fibrations are complemented with the results of \cite{HuTh22_unpub} which identify the pullback of a spherical fibration over a connected sum. The fibration splits over the loop space to produce loop space decompositions. 

\begin{mysubsection}{Connected sum with highly connected manifolds}  
Given $2n$-dimensional Poincar\'{e} duality complexes $M$ and $N$ in which the attaching maps of the $2n$-cells are given by (respectively for $M$ and $N$) 
\[ f : S^{2n-1} \to M_0,~~ g: S^{2n-1} \to N_0,\]
the connected sum is defined as the mapping cone of \cite{Wal67}
\[f+g : S^{2n-1} \to M_0 \vee N_0.\]
Theriault \cite{Th22_unpub} has provided a method to transfer a loop space decomposition of $M$ to one of $N\# M$ for arbitrary $N$. The hypothesis on $M$ used to make the argument work is that the cell attachment is inert. 
\begin{defn}\label{inert}
For a homotopy cofibration sequence 
\[ \Sigma A \stackrel{f}{\to} X_0 \stackrel{i}{\to} X, \]
$f$ is said to be {\it inert} if $\Omega i$ has a homotopy right inverse.
\end{defn}

Let $M_k \in \PP \DD_{n-1}^{2n}$. For $M_k$, we have two different approaches for loop space decompositions \cite{BeTh14} and \cite{BaBa18}. In this case $M_0 \simeq (S^n)^{\vee k}$ and the attaching map is $L(M_k) : S^{2n-1} \to M_0$. As a consequence of the loop space decompositions of $M_k$, we readily observe 
\begin{prop}\label{inerthigh}
For $k\geq 2$, the map $L(M_k)$ is inert. 
\end{prop}
This follows because $\Omega M_k$ is a product of loop spaces of spheres which are mapped via Whitehead products of the sphere inclusions into $M_0$ \cite{BaBa18}. These clearly lift to $\Omega M_0$. Now \cite[Theorem 1.4]{Th22_unpub} implies 
 \begin{theorem}\label{loopdecMconN}
 Let $M\in \PP \DD_{n-1}^{2n}$ with $\Rank(H_n(M))\geq 2$,   and $N\in \PP\DD_1^{2n}$. Then, 
 \[\Omega (M \# N) \simeq \Omega M \times \Omega (\Omega M \ltimes N_0),\]
  where $N_0\simeq N-\ast.$ 
 \end{theorem}

Now by \cite[Theorem 1.4]{BeTh14}, we have 
\[\Omega M \simeq \Omega (S^{n}\times S^{n})\times \Omega (J\vee (J \wedge \Omega (S^{n}\times S^{n}))) \] 
where $J= \vee_{2}^{k}S^{n} $ (where $H_n(M)\cong \Z^k$). Hence we have a decomposition of the loop space $\Omega (M \# N)$ in terms of simply connected spheres if $N_0$ is also a wedge of spheres.

\end{mysubsection}

\begin{mysubsection}{Spherical fibrations over connected sums with highly connected manifolds}
We now construct spherical fibrations over connected sums using the spherical fibrations $S^{n-1} \to E_k \to M_k$ proved in Theorems \ref{sphfibloc}, \ref{sc4mfld}, \ref{sphfib4}, and \ref{sphfibgen}. Let $M_k\in \PP\DD_{n-1}^{2n}$  with $\Rank(H_n(M_k))=k\geq 2$. For a $2n$-dimensional Poincar\'{e} duality complex $N$, we consider the quotient $M_k\# N \stackrel{q}{\to} M_k$,
and the pullback $E_{k,N}$ of  $E_k$ to $M_k\# N$. 
\begin{myeq}\label{pbsphfib}
\xymatrix@C4.5pc{
                           E_{k,N}\ar[r] \ar[d] & E_k \ar[d]                      \\ 
M_k\# N \ar[r]^{q} &  M_k                             ,
}
\end{myeq}
We additionally observe that the loop space of $E_{k,N}$ depends only on $N_0$ and not on the attaching map of the top cell of $N$. 

The homotopy type of $E_{k,N}$ is determined analogously as in \cite[Lemma 3.1]{HuTh22_unpub}. Let $F(n)$ denote the set of homotopy equivalences of $S^{n-1}$. For a $\tau  : S^{n-1} \to F(2n)$ define $\GG_\tau (N)$ as the pushout 
\begin{myeq} \label{Gtaudefn}
\xymatrix@C4.5pc{
S^{2n-1}\times S^{n-1}\ar[r]^{i}\ar[d]^{(\phi , \pi_{2})} & S^{2n-1}\times \DD^n \ar[d]  \\ 
N_{0}\times S^{n-1} \ar[r]^{} &  \GG_{\tau}(N),                             
}
\end{myeq}
 where $\phi$ is the composite
 \[S^{2n-1}\times S^{n-1}\stackrel{t}{\to}  S^{2n-1}\times S^{n-1}\stackrel{\pi_{1}}{\to} S^{2n-1} \to N_{0}, \mbox{  with }t(x,s)=  (\tau(s)x, s).\]
 Using this notation, the pullback $E_{k,N}$ in \eqref{pbsphfib} may be simplified using \cite[Lemma 3.1]{HuTh22_unpub}.
  \begin{prop}\label{pdident}
 There is a $\tau : S^{n-1} \to F(2n) $ such that one has the equivalence 
 \[E_{k,N} \simeq \GG_\tau(N)\# E_k,\]
 with $\GG_\tau(N)$ defined as in \eqref{Gtaudefn}. 
 \end{prop}

 Looking towards the loop space, we build up to an analogue of Theorem \ref{loopdecMconN}. For this, we require the knowledge of the homotopy type of $\GG_\tau(N)-\ast$ for $\ast \in \GG_\tau(N)$. The following identification implies that this is independent of $\tau$ and the attaching map of the top cell. 
\begin{prop} \label{gtauminpt}
 Let $\tau : S^{n-1}\to F(2n)$ be a map. Then, we have a homotopy equivalence 
\[ \GG_\tau(N) -  * \simeq (N_{0}\rtimes S^{n-1}).\]
\end{prop}
\begin{proof}
   Consider the (homotopy) pushout square \eqref{Gtaudefn}.    Express $S^{n-1}= S^{n-1}_{U}\cup S^{n-1}_{L}$ as the union of the upper and lower hemispheres, and  the unit $n$-disk as the union of the disk of radius $1/2$ and the annulus as   $\DD^{n}= \DD^{n}_{\leq 1/2}\cup \DD^{n}_{\geq 1/2}.$ 
We may then write 
\[S^{2n-1}\times \DD^{n}= S^{2n-1}\times \DD^{n}_{\geq 1/2}\cup S^{2n-1}_{U}\times \DD^{n}_{\leq 1/2} \cup S^{2n-1}_{L}\times \DD^{n}_{\leq 1/2} .\]
Hence, 
\[(S^{2n-1}\times \DD^{n}) - \inter(S^{2n-1}_{L}\times \DD^{n}_{\leq 1/2}) \cong S^{2n-1}\times \DD^{n}_{\geq 1/2}\cup S^{2n-1}_{U}\times \DD^{n}_{\leq 1/2}.\]
  Note that $S^{2n-1}_{L}\times \DD^{n}_{\leq 1/2}\cong \DD^{3n-1}$. We set ${\GG_{\tau}(N)}_{0}= \GG_{\tau}(N) - \inter(S^{2n-1}_{L}\times \DD^{n}_{\leq 1/2})$ which is homotopy equivalent to $\GG_\tau(N)- \ast$. The pushout diagram \eqref{Gtaudefn} induces the following homotopy pushout 
\[ 
\xymatrix@C4.5pc{
 S^{2n-1}\times S^{n-1}\ar[r]^-{i}\ar[d]^{(\phi , \pi_{2})} & S^{2n-1}\times \DD^n_{\geq 1/2}\cup  S^{2n-1}_U \times \DD^n_{\leq 1/2} \ar[d]                      \\ 
N_{0}\times S^{n-1} \ar[r]^{} &  \GG_\tau(N)_{0},                             
}
\]
which in turn induces the follwing diagram 
\[ 
\xymatrix@C4.5pc{
 S^{2n-1}_U\times S^{n-1}\ar[r]^{i}\ar[d]^{(i\times I)} &   S^{2n-1}_U \times \DD^n_{\geq 1/2} \ar[d]                    \\ 
 S^{2n-1}\times S^{n-1}\ar[r]^-{i}\ar[d]^{(\phi , \pi_{2})} & S^{2n-1}\times \DD^n_{\geq 1/2}\cup  S^{2n-1}_U \times \DD^n_{\leq 1/2} \ar[d]                      \\ 
N_{0}\times S^{n-1} \ar[r]^{} &  \GG_\tau(N)_{0}   .        
}
\]
Note that the top square is a homotopy pushout square, and hence the outer square 
\[ 
\xymatrix@C4.5pc{
 S^{2n-1}_U\times S^{n-1} \ar[r]^{i}\ar[d]^{(\phi, \pi_{2})\circ (i\times I)} &  S^{2n-1}_U \times \DD^n_{\leq 1/2}\ar[d]                      \\ 
N_{0}\times S^{n-1} \ar[r]^{} &  {\GG_\tau(N)}_{0}                             
}
\]
is also a homotopy pushout square. As $ S^{2n-1}_U \times \DD^n_{\leq 1/2}$ is contractible, ${\GG_\tau(N)}_0$ is homotopy equivalent to the homotopy cofibre of the left vertical arrow, which is easily computed to be $N_0 \rtimes S^{n-1}$ .
\end{proof}

As a consequence of Proposition \ref{gtauminpt}, we obtain the following corollary using \cite[Theorem 1.4]{Th22_unpub}.
\begin{cor}\label{loopdecgtau}
  Suppose $E \in \PP \DD_1^{3n-1}$ in which the attaching map of the top cell is inert and $N\in \PP \DD_1^{2n}$. Then, for any $\tau: S^{n-1} \to F(2n)$, 
 \[\Omega (\GG_\tau(N)\# E)\simeq \Omega E \times \Omega (\Omega E \ltimes J )\]
  where $J=  (N_{0}\rtimes S^{n-1})$. 
\end{cor}

\end{mysubsection}

\begin{mysubsection}{Loop space decompositions from spherical fibrations}\label{loopspdec}
We obtain loop space decompositions for $E_{k,N}$ using Corollary \ref{loopdecgtau} and Proposition \ref{pdident}. Note that the calculations of \cite{BaBa18} implies that the attaching map of the top cell of a connected sum of sphere products is inert. Therefore, we deduce the following result by applying Corollary \ref{loopdecgtau} for $E=E_k$.

\begin{prop}\label{decgtauconek}
Let $N\in \PP \DD_1^{2n}$. Then, for any $\tau: S^{2n-1} \to F(n)$,
\[ \Omega (\GG_\tau(N)\# E_k)\simeq \Omega E_k \times \Omega (\Omega E_k \ltimes (N_0 \rtimes S^{n-1}) ).\]
Further, one has the equivalence $\Omega E_k \simeq \Omega (S^{n}\times S^{2n-1})\times \Omega(J\vee J \wedge \Omega (S^{n}\times S^{2n-1}))$,  where $J\simeq (S^{n}\vee S^{2n-1})^{\vee k-2}.$
\end{prop}

\begin{theorem}\label{loopspdecconnsum}
   Suppose $M_k\in \PP \DD_{n-1}^{2n}$  with $H_{n}M_k\cong \Z^k$ and $k\geq 2$, and $N\in \PP \DD_1^{2n}$. Let $E_k= \#^{k-1}(S^{n}\times S^{2n-1})$ for an even integer $n$.  Then, we have the homotopy equivalence 
   $$\Omega(N\# M_k)\simeq S^{n-1}\times  \Omega E_k \times \Omega (\Omega E_k \ltimes Y )$$
    where  $Y\simeq (N_{0}\rtimes S^{n-1}) , $ after inverting the primes in $T_n$ (that is, those occuring in the torsion part of $\pi_{2n-1}(S^n)$ together with the prime $2$).
\end{theorem}
\begin{proof}
    The proof follows directly from Theorem \ref{sphfibloc} and Propositions \ref{decgtauconek} and \ref{pdident}. 
 
\end{proof}

\begin{theorem}\label{connloopdechigh}
   Suppose $M_k \in \PP \DD_{n-1}^{2n}$ with rank ($H_{n}M_k)= k > r,$ where $r= \#$cyclic torsion summands in $\pi_{n-1}^s$, and $N\in \PP\DD_1^{2n}$ .  Then after inverting   $2$,  we have the homotopy equivalence 
$$\Omega(N\# M_k)\simeq S^{n-1}\times  \Omega E_k \times \Omega (\Omega E_k \ltimes Y )$$ 
where  $Y\simeq (N_{0}\rtimes S^{n-1})$. 
\end{theorem}
\begin{proof}
     The proof follows directly from the Theorem \ref{sphfibgen} and  Propositions \ref{decgtauconek} and \ref{pdident}. 
\end{proof}

For $n=2,~4$, we do not have to invert $2$, and we have
\begin{theorem}\label{connloopdec24}
   Suppose $M_k \in \PP \DD_{n-1}^{2n}$ with rank ($H_{n}M_k)= k \geq 2,$ and $n\in \{2,4\}$, and $N\in \PP\DD_1^{2n}$ .  Then,  we have the homotopy equivalence 
$$\Omega(N\# M_k)\simeq S^{n-1}\times  \Omega E_k \times \Omega (\Omega E_k \ltimes Y )$$ 
where  $Y\simeq (N_{0}\rtimes S^{n-1})$. 
\end{theorem} 
\end{mysubsection}

\begin{mysubsection}{Decomposition of looped configuration spaces} \label{loopconfig}
The loop space decomposition of a space has many applications, one of them being in the case of configuration spaces. We note this down in the examples treated above. Recall that the ordered configration space of $X$ is given by $F_{k}(X)= \{(x_{1}, \dots, x_{k})\mid x_{i}\neq x_{j}$ if $i\neq j\}.$
\begin{defn}\cite{CoGi02}
  Let $\pi \colon F_{k}(M) \rightarrow M$ be the projection onto the first factor. The space $M$ is said to be a $\sigma_{k}$-manifold if $\pi$ admits a  cross section.
\end{defn}
If $M$ is a $\sigma_k$-manifold, \cite[Theorem 2.1]{CoGi02} implies the homotopy equivalence 
\begin{myeq}\label{configdecomp}
\Omega F_{k}(M)\simeq \Omega M \times \Omega (M-Q_{1}) \times \dots \times \Omega (M-Q_{k}),
\end{myeq}
for any choice of $k$ distinct points $q_{1},\cdots , q_{k}$ of $M$ with $Q_{i}= \{q_{1},\cdots , q_{i}\}$. The hypothesis of being a $\sigma_k$-manifold is satisfied if $M$ has a nowhere vanishing vector field \cite[Theorem 5]{FaNe62}. By the Poincar\'{e}-Hopf index theorem, this is satisfied if the Euler characteristic $\chi(M)=0$. Putting all this together we obtain the following result for $N\# M$ where $M$ is a $(n-1)$-connected $2n$-manifold. 
\begin{theorem}\label{configloopdec}
    Suppose $M$ is an $(n-1)$-connected $2n$-manifold for $n$ even such that $\Rank(H_{n}(M))=r\geq 2$, and $N$ is a simply connected $2n$-manifold  with $\chi(N)= -r$. Then, after inverting the primes in $T_n$, the homotopy type of $\Omega F_{k}(N\#M)$ depends only on the homotopy type of $N-*$ and the integer $r$. More precisely we have the decomposition 
   $$\Omega F_{k}(N\#M)\simeq S^{n-1}\times  \Omega E \times \Omega (\Omega E \ltimes Y ) \times \Omega (N_{0}\vee (S^{n})^{\vee_{r}})\times \prod_{i=1}^{k-1} \Omega (N_{0}\vee (S^{n})^{\vee_{r}}\vee (S^{2n-1})^{\vee_{i}})$$
   in which     $Y\simeq (N_{0}\rtimes S^{n-1}) , E= \#^{r-1}(S^{n}\times S^{2n-1})$.
   \end{theorem}
   \begin{proof}
The hypothesis $\chi(N)=-r$ implies $\chi(M\# N)=0$. Thus, \eqref{configdecomp} applies to give 
$$\Omega F_{k}(N\#M)\simeq \Omega (N\#M) \times \Omega (N\#M -Q_{1}) \times \cdots \times \Omega (N\# M -Q_{k}) $$ 
We now have the equivalence 
\[\Omega (N\# M) \simeq S^{n-1}\times  \Omega E \times \Omega (\Omega E \ltimes Y )\]
by Theorem \ref{loopspdecconnsum}. The other factors in the product are observed via the equivalences 
\[N\# M - Q_i \simeq (N\#M - \mbox{pt}) \vee (S^{2n-1})^{\vee i-1} \simeq N_0 \vee (S^n)^{\vee r} \vee (S^{2n-1})^{\vee i-1}\]
for $i\geq 1$.
   \end{proof}
  Finally, one may observe that slightly more general versions of Theorem \ref{configloopdec} are provable using Theorems \ref{connloopdechigh} and \ref{connloopdec24}.
   \end{mysubsection}

\end{document}